\documentclass[12pt,twoside]{amsart}

\usepackage[margin=1in]{geometry}
\usepackage{amssymb}
\usepackage{physics}
\usepackage{tensor}

\usepackage{todonotes}

\usepackage[]{hyperref}
\usepackage[initials]{amsrefs}
\DefineSimpleKey{bib}{how}
\renewcommand{\eprint}[1]{#1}

\BibSpec{misc}{
  +{}{\PrintAuthors}  {author}
  +{,}{ \textit}      {title}
  +{.}{ }             {how}
  +{}{ \parenthesize} {date}
  +{,} { available at \eprint}        {eprint}
  +{,}{ available at \url}{url}
  +{,}{ }             {note}
  +{.}{}              {transition}
}
\numberwithin{equation}{section}

\usepackage{amsthm}

\theoremstyle{plain}
\newtheorem{theorem}{Theorem}[section] %[section]
\newtheorem{lemma}[theorem]{Lemma}
\newtheorem{corollary}[theorem]{Corollary}
\newtheorem{proposition}[theorem]{Proposition}

\newtheorem{theoremAlph}{Theorem}

\theoremstyle{definition}
\newtheorem{definition}[theorem]{Definition}

\newtheorem{example}[theorem]{Example}

\theoremstyle{remark}
\newtheorem{remark}[theorem]{Remark}

\newcommand{\bC}{\mathbb C}
\newcommand{\bN}{\mathbb N}
\newcommand{\bR}{\mathbb R}
\newcommand{\bT}{\mathbb T}
\newcommand{\bZ}{\mathbb Z}
\newcommand{\cA}{\mathcal A}
\newcommand{\cB}{\mathcal B}

\newcommand{\cH}{\mathcal H}

\newcommand{\cM}{\mathcal M}
\newcommand{\cN}{\mathcal N}
\newcommand{\cO}{\mathcal O}

\newcommand{\cU}{\mathcal U}

\newcommand{\id}{\mathrm{id}}
\newcommand{\opo}{\mathrm{op}}
\newcommand{\SO}{\mathrm{SO}}

\DeclareMathOperator{\Aut}{Aut}
\DeclareMathOperator{\Rep}{Rep}
\DeclareMathOperator{\Irr}{Irr}
\DeclareMathOperator{\Sym}{Sym}
\DeclareMathOperator{\Ad}{Ad}

\title[Crossed Product Equivalence of Quantum Automorphism Groups]{Crossed Product Equivalence of Quantum Automorphism Groups of Finite Dimensional C$^\ast$-algebras}

\author{Michael Brannan}

\address{Michael Brannan\\ Department of Pure Mathematics and The Institute for Quantum Computing \\ University of Waterloo\\
Waterloo, ON\\
N2L 3G1 Canada}
\email{michael.brannan@uwaterloo.ca}

\author{Floris Elzinga}
\address{Floris Elzinga \\ Department of Mathematics \\
University of Oslo\\
P.O box 1053, Blindern \\
0316 Oslo, Norway}
\email{florise@math.uio.no}

\author{Samuel J. Harris}
\address{Samuel J. Harris \\ Department of Mathematics \& Statistics\\
Northern Arizona University \\
PO Box 5717, Flagstaff, AZ \\
86011 USA}
\email{samuel.harris@nau.edu}

\author{Makoto Yamashita}
\address{Makoto Yamashita \\ Department of Mathematics \\
University of Oslo\\
P.O box 1053, Blindern \\
0316 Oslo, Norway}
\email{makotoy@math.uio.no}

\begin{document}

\begin{abstract}
We compare the algebras of the quantum automorphism group of finite-dimensional C$^\ast$-algebra $B$, which includes the quantum permutation group $S_N^+$, where $N = \dim B$.
We show that matrix amplification and crossed products by trace-preserving actions by a finite Abelian group $\Gamma$ lead to isomorphic $\ast$-algebras.
This allows us to transfer various properties such as inner unitarity, Connes embeddability, and strong $1$-boundedness between the various algebras associated with these quantum groups.
\end{abstract}
\maketitle

\section{Introduction}

Given a finite dimensional C$^\ast$-algebra $B$ equipped with a faithful state $\psi$, Wang constructed in \cite{Wan98} the \emph{quantum automorphism group} $\Aut^+(B,\psi)$ of the finite measured quantum space $(B,\psi)$.
By construction, it is a C$^\ast$-algebraic compact quantum group whose underlying Hopf $\ast$-algebra $\cO(\Aut^+(B,\psi))$ is defined to be the universal $\ast$-algebra generated by the coefficients of a $\ast$-coaction on $B$ that leaves $\psi$ invariant.
In particular, when we consider the ``Plancherel'' trace  $\psi$ on $B$, the canonical tracial state invariant under the classical automorphism group $\Aut(B)$, then $\Aut^+(B) = \Aut(B, \psi)$ has a close analogy with $\Aut(B)$, as the Abelianization of $\cO(\Aut^+(B))$ becomes the function algebra $\cO(\Aut(B))$.

There are two extreme choices for $B$: one one hand is the abelian $B = \bC^N$, on the other is the full matrix algebra $B = M_n$.
In the former case, $\Aut^+(\bC^N)$ is the \emph{quantum permutation group} $S_N^+$ that has a close analogy to the permutation group $S_N$ and is accessible via various combinatorial methods.
In the latter case, $\Aut^+(M_n)$ can be identified with the projective version of the quantum orthogonal group $O_n^+$ \cite{MR1709109}, which, besides combinatorial methods, also allows analogy with classical functional analysis on the orthogonal group $O_n$. This kind of correspondence is what we are going to exploit in this work.

In the framework of operator algebraic quantum groups, operator algebraic completions of $\cO(\Aut^+(B))$ are known to have various interesting properties.
In general, if $\dim B \geq 5$, it is known that the reduced C$^\ast$-algebra $C^r(\Aut^+(B))$ is non-nuclear, exact, simple, with unique trace, and possesses the complete metric approximation property, while the von Neumann algebra $L^\infty(\Aut^+(B))$ is a non-injective, weakly amenable, strongly solid II$_1$-factor with the Haagerup property \cites{MR3138849,DeFrYa14,GWPrep,MR3391904}.
Moreover, $\cO(S_N^+)$ is residually finite-dimensional and consequently $L^\infty(S_N^+)$ has the Connes Embedding Property (CEP), for all $N$ \cite{BrChFr18}.

The above results lead to a question: How much do the operator algebras $C^r(\Aut^+(B))$ and $L^\infty(\Aut^+(B))$ actually depend on the initial data $B$?
One key tool common in the above works is the C$^\ast$-tensor category of finite dimensional unitary representations of these quantum groups, and induction of algebraic properties through monoidal equivalence.
In fact, the monoidal equivalence classes of the quantum groups $\Aut^+(B)$ are classified by the dimension of $B$ \cite{DeVa10}, hence one may hope that the operator algebras $C^r(\Aut^+(B))$ or $L^\infty(\Aut^+(B))$ may be closely related (possibly even isomorphic) as we range over $B$ with $\dim B$ fixed.
In particular, it is natural to ask if monoidal equivalence can be used to transfer the CEP from $L^\infty(S_N^+)$ to all $L^\infty(\Aut^+(B))$.

At the C$^\ast$-algebraic level, even when $\dim B$ is fixed, Voigt \cite{MR3717094} showed that  $C^r(\Aut^+(B))$ do depend on the choice of fiber functors realizing these quantum groups out of the common category $\Rep(\Aut^+(B)) = \Rep(S_N^+)$.
Nonetheless, it is an interesting question to ask to what degree the algebras above differ, on either the C$^\ast$-algebraic or the von Neumann algebraic level.

Our main result is that, up to crossed products by finite Abelian groups and matrix amplification, there are some concrete relations between the algebras of these quantum groups.

\begin{theoremAlph}[Corollary \ref{cor:dbl-cross-prod-is-mat-amp}, Theorem \ref{T: crossed product isomorphism}] \label{THA}
Let $N = \dim B \geq 4$.
Then there is a finite Abelian group $\Gamma$, a trace-preserving action $\alpha$ of $\Gamma$ on $\cO(\Aut^+(B))$, and another trace-preserving action $\beta$ of $\Gamma$ on the crossed product $\cO(\Aut^+(B)) \rtimes_\alpha \Gamma$, such that we have an isomorphism of tracial $\ast$-algebras
\begin{equation*}
    \cO(\Aut^+(B)) \rtimes_\alpha \Gamma \rtimes_\beta \Gamma \cong M_k \otimes \cO(S_N^+).
\end{equation*}
\end{theoremAlph}

We prove two versions of this result for different values of $k = \operatorname{ord}(\Gamma)$.
Let us write $B = \bigoplus_{r=1}^m M_{n_r}$, and set $d = \prod_{r=1}^m n_r$.
Using cocycle deformation of Hopf algebras and coactions, we obtain the above result for $k = d^4$ from a $2$-cocycle induced from a finite subgroup of $S_N$.
However, we also show that the more efficient value $k = d^2$ is achievable using techniques inspired by non-local games in quantum information theory, in particular a certain quantum colouring game of the so-called quantum complete graph $K_B$ \cite{BGH20}.

Both of these ideas lead to the construction of certain concrete finite dimensional representations of the linking algebra $\cO(\Aut^+(B),S_N^+)$ associated to the monoidal equivalence between $\Aut^+(B)$ and $S_N^+$.
Combining this with the standard induction argument, we can transfer finite dimensional approximation results on $S_N^+$ \cite{BrChFr18} to all quantum automorphism groups.

\begin{theoremAlph}[Corollary \ref{cor:Aut-RFD-and-CEP}]\label{THB}
Let $B$ be a finite dimensional C$^\ast$-algebra and $\psi$ be a faithful tracial state on $B$.
Then the Hopf $\ast$-algebra $\cO(\Aut^+(B,\psi))$ is residually finite-dimensional and the von Neumann algebra $L^\infty(\Aut^+(B))$ has the CEP.
\end{theoremAlph}

In fact, an even stronger form of residual finite-dimensionality holds for $\cO(\Aut^+(B))$. 
Recall that the Hopf $\ast$-algebra $\cO(G)$ of a compact quantum group is called \emph{inner unitary} if it admits an inner faithful $\ast$-homomorphism into some $M_k$ \cites{bb-inner, BrChFr18}.
Inner unitarity of $\cO(G)$ is a quantum generalization of the property of discrete group $\Gamma$ admitting an embedding into a unitary group $U_k \subseteq M_k$.  Indeed,  if $\cO(G) = \bC \Gamma$ for some discrete group $\Gamma$, then $\cO(G)$ is inner unitary if and only if an embedding $\Gamma  \hookrightarrow U_k$ exists. In general, $\cO(G)$ is residually finite-dimensional if it is inner unitary.

\begin{theoremAlph}[Proposition \ref{prop:Aut-Inner-Unitary}, Theorem \ref{thm:FO-Inner-Unitary}]\label{THC}
Let $B$ be a finite-dimensional C$^\ast$-algebra.  Assume that $\dim(B)$ lies outside the range $[6,9]$.
Then $\cO(\Aut^+(B))$ is inner unitary.  The same conclusion also holds for the Hopf $\ast$-algebras $\cO(O_n^+)$ for $n \ne 3$, where $O_n^+$ is the free orthogonal quantum group. 
\end{theoremAlph}

\medskip
\emph{Strong $1$-boundedness} is an important free probabilistic property of tracial von Neumann algebras $\cM$ introduced by Jung \cite{MR2373014}, that is computable from any finite generating set $\cU \subset \cM$ and is a strengthened version of $\delta_0(\mathcal{U}) \le 1$ for Voiculescu's modified free entropy dimension \cite{MR1371236}.
In particular, this is an obstruction to isomorphism with another tracial von Neumann algebra with a generating set satisfying $\delta_0(\cU) > 1$, such as an interpolated free group factor.

While $\delta_0(\cU)$ could be difficult to compute in general, a useful estimate is given in terms of  \emph{$\ell^2$-Betti numbers}.
Generalizing an estimate for the von Neumann algebras of discrete groups \citelist{\cite{MR2889142}\cite{MR2180603}}, when $G$ is a compact matrix quantum group and $\cU$ is the standard set of generators coming from the fundamental representation, one has \cite{BCV}
\begin{equation} \label{BCV-estimate}
1 \leq \delta_0(\mathcal{U}) \leq 1 - \beta_0^{(2)}(\hat G) + \beta_1^{(2)}(\hat G),
\end{equation}
where $\beta_k^{(2)}(\hat G)$ are the $\ell^2$-Betti numbers of the discrete dual of $G$ defined by Kyed \cite{MR2464704}.

For our quantum groups of interest, we have the vanishing of $\beta_0^{(2)}(\hat G)$ and $\beta_1^{(2)}(\hat G)$ \citelist{\cite{MR1975007}\cite{MR3049699}\cite{MR3683927}}, hence the standard generators satisfy $\delta_0(\cU) = 1$.
To upgrade this to the strong $1$-boundedness, one needs to work with more precise algebraic relations, and estimate regularity and rank of the induced operators \citelist{\cite{arXiv:1602.04726}\cite{MR4198970}}.
This was successfully carried out by the first author and Vergnioux for $O_n^+$ \cite{MR3783410}, and by the second author for the quantum orthogonal group $O_J^{+}$ associated with the symplectic matrix \cite{MR4288353}.

Based on these results and our main results, we can now prove the strong $1$-boundedness of $\Aut^+(B)$ for some cases, as follows.

\begin{theoremAlph}[Corollary \ref{col:S1B-QAG}]\label{THE}
Let $B$ be a C$^\ast$-algebra such that $\dim B = n^2$ with $n \geq 3$.
Then $L^\infty(\Aut^+(B))$ is a strongly $1$-bounded II$_1$-factor.
\end{theoremAlph}

The starting point is the index $2$ inclusion $L^\infty(\Aut^+(M_n)) \subset L^\infty(O_n^+)$.
Moreover, Theorem \ref{THA} gives rise to finite index embeddings into the common overfactors of the von Neumann algebras $L^\infty(\Aut^+(B))$ with fixed $\dim B$.
Thus, the remaining task is to obtain permanence of strong $1$-boundedness under such relations.

In general, given a finite index inclusions of II$_1$-factors $\cN \subset \cM$, one expects
\begin{align*}
    \delta_0(X) - 1 = [\cM:\cN] \left( \delta_0(Y) - 1 \right)
\end{align*}
for generating sets $X$ for $\cN$ and $Y$ for $\cM$ as an analogue of the Nielsen--Schreier Theorem for inclusions of free groups.
This theorem states that a finite index subgroup $H$ of a free group $G$ of rank $r(G)$ must also be free, and moreover that the rank $r(H)$ of the subgroup satisfies
\begin{align*}
    r(H) - 1 = [G:H] \left( r(G) - 1 \right) .
\end{align*}
This was investigated by Jung in \cite{arXiv:math/0410594}, where he proved, among other things, that this equality holds if $\cM = \cN \otimes M_k$ for some $k$ and $Y$ is $X$ together with the matrix units, but in general only weaker inequalities were established.

Notice that Schreier's formula above suggests that having free entropy dimension equal to $1$ is preserved by finite index inclusions.
We can prove this statement under the stronger assumption of strong $1$-boundedness. The following result was also independently obtained by  Srivatsav Kunnawalkam Elayavalli (private communication). 

\begin{theoremAlph}[Theorem \ref{S1B-SFs}] \label{THD}
Assume that $\cN\subset \cM$ is a unital finite index inclusion of II$_1$-factors. Then $\cN$ is strongly $1$-bounded if and only if $\cM$ is strongly $1$-bounded.
\end{theoremAlph}

As a consequence of Theorems \ref{THA} and \ref{THD}, we can prove that an infinite family of quantum automorphism groups give rise to new examples of strongly $1$-bounded II$_1$-factors which have neither Property $\Gamma$ nor Property (T).

We also briefly consider the free unitary quantum groups $U_n^+$.
The estimate \eqref{BCV-estimate} becomes $1 \le \delta_0(\cU) \le 2$ for this case  \cite{MR3784753}, and the isomorphism $L^\infty(U_2^+) \cong \mathcal{L} \mathbb F_2$ \cite{Ba97} suggests that equality with the upper bound is likely to happen.
While we cannot quite verify this, we can show that $L^\infty(U_n^+)$ is not strongly $1$-bounded (Proposition \ref{P:nots1b}).
In particular, this implies the non-isomorphism result $L^\infty(U_n^+) \ncong L^\infty(O^+_m)$ for any $n , m \ge 2$.

\subsection{Outline of the Paper}

In Section \ref{sec:Preliminaries} we briefly review the necessary material on compact quantum groups (in particular the quantum automorphism groups) and monoidal equivalences between them.

We give the first proof of Theorem \ref{THA} using cocycle deformation in Section \ref{sec:IsoFromCocycleDeformations}.
Section \ref{sec:mon-equiv-mat-model} contains transference results for monoidal equivalence when the linking algebra admits a finite dimensional representation and applies these results to the embeddings in Theorem \ref{THA}.
This establishes Theorem \ref{THB} and Theorem \ref{THC}.
We continue with applications of Theorem \ref{THA} to strong $1$-boundedness in Section \ref{sec:S1B}.
We recall the definition of strong $1$-boundedness and some results from the theory of subfactors before establishing Theorem \ref{THD}, the main technical result of the section, and proceed to derive Theorem \ref{THE}.
Additionally we discuss lack of strong $1$-boundedness for the free unitary quantum groups.

In the final Section \ref{sec:unitary-error-basis} we provide another proof of Theorem \ref{THA} using ideas from unitary error bases and the theory of non-local games, which we first recall.
This proof leads to a more efficient version of Theorem \ref{THA}, with a smaller acting group and lower dimensional matrix algebras.

\subsection*{Acknowledgements} Brannan and Harris were partially supported by NSF Grant DMS--2000331. Harris was partially supported by an NSERC Postdoctoral Fellowship.
Elzinga and Yamashita were partially supported by the NFR funded project 300837 ``Quantum Symmetry''.
We also thank Mathematisches Forschungsinstitut Oberwolfach for hosting the workshop `Quantum Groups - Algebra, Analysis and Category Theory' in September 2021, where our collaboration in its final form started.
The authors are indebted to Julien Bichon for his insightful comments on an earlier version of this work and for pointing out the applications to inner unitary Hopf $\ast$-algebras, as well as connections to cocycle twists.
The authors also thank Srivatsav Kunnawalkam Elayavalli for useful insights on strongly $1$-bounded von Neumann algebras.

\section{Preliminaries}
\label{sec:Preliminaries}
\subsection{Compact Quantum Groups}

For the basic theory of compact quantum groups and their representation categories, we refer to the book \cite{MR3204665}.

 \begin{definition}
 A {\it compact quantum group} $G$ consists of a unital Hopf $\ast$-algebra $\cO(G)$ with coproduct $\Delta \colon \cO(G) \to \cO(G) \otimes \cO(G)$ together with a {\it Haar functional}, which is a linear map $h \colon \cO(G) \to \bC$ satisfying the following properties:
 \begin{itemize}
     \item $h$ is {\it invariant} in the sense that $(\iota \otimes h) \Delta(x) = h(x) 1 = (h \otimes \iota) \Delta(x)$ for all $x \in \cO(G)$;
     \item $h$ is normalized such that $h(1) = 1$;
     \item For any $x \in \cO(G)$ it holds that $h(x^* x) \geq 0$.
 \end{itemize}
 \end{definition}

We can associate two reduced operator algebras to $G$ using the Haar functional in the obvious way.
The {\it reduced C$^\ast$-algebra} $C^r(G)$ is the C$^\ast$-algebra completion of $\cO(G)$ relative to the GNS representation induced by the Haar functional $h$, and the {\it von Neumann algebra of $G$}, $L^\infty(G)$, is the von Neumann algebra $C^r(G)''$ generated by $C^r(G)$.
If the Haar functional is a trace, we say that $G$ is of {\it Kac type}.
Note that in this case $L^\infty(G)$ is a finite von Neumann algebra, and we consider (the canonical normal extension of) $h$ to be the canonical trace on $L^\infty(G)$.

A {\it unitary representation} of $G$ on a finite dimensional Hilbert space $\cH$ is a unitary corepresentation of $\cO(G)$ on $\cH$.
It is well-known that the category of finite dimensional unitary representations of $G$, denoted Rep$(G)$, is a rigid C$^\ast$-tensor category.
A {\it right action} of $G$ is a right coaction of $\cO(G)$ on some $\ast$-algebra $A$.
More precisely, this is a $\ast$-homomorphism $\delta \colon A \to A \otimes \cO(G)$ such that $(\iota \otimes \Delta) \delta = (\delta \otimes \iota) \delta$ and $(\iota \otimes \varepsilon) \delta = \iota$.
Left actions are defined similarly.  Given such an action of $G$ on $A$, we can define the algebra of invariant elements $A^G$ as consisting of those $a \in A$ that satisfy $\delta(a) = a \otimes 1$.
If it happens that $A^G = \bC 1$, we say that the action is {\it ergodic}.
The action is called {\it free} if the linear map $A \otimes A \to A \otimes \cO(G)$ given by $a \otimes b \mapsto \delta(a)(b \otimes 1)$ is invertible.

There are many interesting examples of compact quantum groups.
To close this section we define the free unitary and free orthogonal quantum groups, and in the next section we discuss in detail the family of examples known as the quantum automorphism groups.

\begin{definition}[\cites{MR1382726,MR1378260}]
Let $n \geq 2$ be an integer and choose complex invertible matrices $Q$ and $F$ of size $n$ such that $Q$ is positive and $F \overline{F} \in \bR I_n$.
We define the following two universal $\ast$-algebras:
\begin{align*}
    \cO(U_Q^+) &= \bC \left\langle v_{ij} \mid 1 \leq i,j \leq n , ~ V = (v_{ij})_{ij} \text{ and } Q\overline{V}Q^{-1} \text{ are unitary} \right\rangle \\
    \cO(O_F^+) &= \bC \left\langle u_{ij} \mid 1 \leq i,j \leq n , ~ U = (u_{ij})_{ij} \text{ is unitary and } U = F\overline{U}F^{-1} \right\rangle .
\end{align*}
The Hopf $\ast$-algebra structure is then defined by
\begin{align*}
    \Delta(w_{ij}) = \sum_{k=1}^n w_{ik} \otimes w_{kj}, \quad (S \otimes \id)W = W^{^*}, \quad (\epsilon \otimes \id)W = 1_n.  
\end{align*}
where $W = [w_{ij}] \in \{V,U\}$.
The resulting compact quantum groups are called the {\it free unitary quantum group} $U_Q^+$ and the {\it free orthogonal quantum group} $O_F^+$, respectively.
\end{definition}

In terms of representation theory, both $U_Q^+$ and $O_F^+$ can be interpreted as the universal compact quantum groups given by defining unitary irreducible representations $V$ and $U$ respectively, with prescribed dual representations.
The matrices $Q$ and $F$ enter the picture to correct for the fact that the contragredient representation to a unitary representation is not automatically unitary in the case of quantum groups.
Instead, it is in general necessary to correct by conjugating by some matrix to obtain the dual representation.
Then $U_Q^+$ is the universal compact quantum group for which that conjugating matrix is precisely $Q$, and the same statement for $F$ holds for $O_F^+$ with the additional demand that its defining representation is self-dual.

If one makes the choice $Q = F = I_n$, it is customary to write $U_{I_n}^+ = U_n^+$ and $O_{I_n}^+ = O_n^+$.
$O_n^+$ and $U_n^+$ are always of Kac type. Moreover, their associated von Neumann algebras are II$_1$-factors and have been extensively studied (see for instance \cites{MR2355067,MR2995437,MR3084500,DeFrYa14,MR3455859,MR3391904,MR4211088}).
The only other choice (up to isomorphism) of $F$ that leads to a Kac type compact quantum group is $O_{J_{2m}}^+$, where $J_{2m}$ is the standard symplectic matrix of size $2m$.
We will denote $O_{J_{2m}}^+ = O_{2m}^{+J}$, as it can be realised as a graded twist of $O_{2m}^+$ \cite{MR3580173}.

\subsection{Quantum Automorphism Groups}

Following Wang, we consider the quantum  automorphism group of a finite measured quantum space $(B,\psi)$, as follows.

\begin{definition}[\cite{Wan98}]
Let $B$ be a finite-dimensional C$^\ast$-algebra $B$ equipped with a faithful state $\psi$.
The {\it quantum automorphism group} $\Aut^+(B,\psi)$ is the compact quantum group with Hopf $\ast$-algebra $\cO(\Aut^+(B,\psi))$ given by the universal unital $\ast$-algebra generated by the coefficients of a coaction 
\[
\rho\colon B \to \mathcal O(\Aut^+(B,\psi)) \otimes B,
\]
satisfying the $\psi$-invariance condition
\[
(\id \otimes \psi)\rho(x) = \psi(x)1 \qquad (x \in B).
\]
By {\it coefficients} of the coaction $\rho$ we mean the set $\{(\omega \otimes \id)\rho(x) \colon x \in B, \ \omega \in B^*\}$
\end{definition}

The Hopf $\ast$-algebra structure of $\cO(\Aut^+(B,\psi))$ is uniquely determined by the above requirements.
For example, the coproduct map
\[
\Delta\colon \cO(\Aut^+(B,\psi)) \to \cO(\Aut^+(B,\psi)) \otimes \cO(\Aut^+(B,\psi))
\]
can be computed from the coaction identity
\[
(\rho \otimes \id)\rho = (\id \otimes \Delta)\rho.
\] 

The quantum group $\Aut^+(B,\psi)$ can be regarded as a universal quantum analogue of the compact group of $\ast$-automorphisms $\Aut(B)$ of $B$.
More precisely, we call an automorphism $\alpha \in \Aut(B)$ {\it $\psi$-preserving} if $\psi \circ \alpha = \psi$.
Denoting the subgroup of all $\psi$-preserving automorphisms by $\Aut(B,\psi) < \Aut(B)$, one sees that the algebra of coordinate functions $\mathcal O(\Aut(B,\psi))$ is precisely the Abelianization of $\mathcal O (\Aut^+(B,\psi))$.

\begin{definition}
Let $B$ be a finite-dimensional C$^\ast$-algebra, and $\delta > 0$.
A \emph{$\delta$-form} on $B$ is a state on $B$ such that $m \circ m^* = \delta \id$, where $m^* \colon B \to B \otimes B$ is the adjoint of the multiplication map $m\colon B \otimes B \to B$ with respect to the Hermitian inner products associated with $\psi$.
\end{definition}

It suffices to understand the quantum automorphism groups $\Aut^+(B,\psi)$ with $\psi$ a $\delta$-form due to the following result, a proof of which can be found in \cite{DeFrYa14}.

\begin{proposition}
\label{P: general trace}
Let $B$ be a finite-dimensional C$^\ast$-algebra $B$ equipped with a faithful state $\psi$.
Fix an isomorphism $B=\bigoplus_{r=1}^m B_r$ corresponding to the coarsest direct sum decomposition with the property that every $B_r$ is a C$^\ast$-algebra and every restriction $\psi |_{B_r}$ becomes a $\delta_r$-form for some $\delta_r$.
Denote by $\psi_r$ the state obtained on $B_r$ by normalizing the restriction $\psi |_{B_r}$.
Then we have an isomorphism
\begin{align}
    \cO(\Aut^+(B,\psi)) \cong \ast_{r=1}^m \cO(\Aut^+(B_r,\psi_r))
\end{align}
of Hopf $\ast$-algebras.
\end{proposition}

\begin{example}
Let $B$ be a finite-dimensional C$^\ast$-algebra, and fix an isomorphism  $B=\bigoplus_{r=1}^m M_{n_r}$.
Its \emph{Plancherel trace} is the tracial state defined by
\[
\psi(A)=\sum_{r=1}^m \frac{n_r}{\dim(B)} \Tr_{n_r}(A_r) \quad (A=\bigoplus_{r=1}^m A_r \in B, \, A_r \in M_{n_r}
).
\]
This is the unique tracial $\delta$-form on $B$, and we have $\delta = \sqrt{\dim B}$.
\end{example}

Since the Plancherel trace $\psi$ is always an invariant state for the action of $\Aut(B)$ on $B$, we have $\Aut(B,\psi) = \Aut(B)$, so $\Aut^+(B,\psi)$ can truly be regarded as the quantum analogue of $\Aut(B)$.
With this in mind, we shall often suppress the $\psi$-dependence in our notation and simply write $\Aut^+(B) = \Aut^+(B,\psi)$ for the remainder of the paper, understanding that the Plancherel trace is used unless specified otherwise.

\subsection{Monoidal equivalence and the linking algebra}

Let $G_1$ and $G_2$ be \emph{monoidally equivalent} compact quantum groups, that is, suppose that there is a unitary monoidal equivalence of C$^\ast$-tensor categories $F \colon \Rep(G_1) \to \Rep(G_2)$.
Such a situation is captured by the associated \emph{linking algebra}, or \emph{Hopf--Galois object} \cites{bichon-lect-note-gal-ext-cpt-qgrp,MR2202309}, which we denote by $\cO(G_1,G_2)$.

This is a unital $\ast$-algebra equipped with a pair of commuting free ergodic coactions
\begin{align*}
\delta_1\colon \cO(G_1,G_2) &\to \cO(G_1) \otimes \cO(G_1,G_2),&
\delta_2\colon \cO(G_1,G_2) &\to \cO(G_1,G_2) \otimes \cO(G_2).
\end{align*}
We denote by $\omega_{12}\colon \cO(G_1,G_2)  \to \bC$ the unique faithful invariant state, which is characterized by
\[
\omega_{12}(x)1 = (h_1 \otimes \id) \delta_1(x) = (\id \otimes h_2)\delta_2(x) \qquad (x \in \cO(G_1,G_2)).
\]
More concretely, $\cO(G_1, G_2)$ has a linear basis of the form $(x_{ij}^\pi)_{\pi, i, j}$ where $\pi$ runs over $\Irr(G_1)$, $i$ runs over an index set for an orthonormal basis of $H_{\pi}$, and $j$ runs over such a set for $H_{F(\pi)}$, with the convention $x^\pi_{i j} = 1$ when $\pi$ is the trivial representation. Note that each (possibly rectangular) matrix $X^\pi = [x_{ij}^\pi] \in M_{d(\pi), d(F(\pi))} \otimes \cO(G_1, G_2)$ is unitary.
Then we can present the above maps as
\begin{align*}
\delta_1(x^\pi_{i j}) &=  \sum_{k = 1}^{\dim(\pi)} u^\pi_{i k} \otimes x^\pi_{k j},&
\delta_2(x^\pi_{i j}) &=  \sum_{k = 1}^{\dim(F(\pi))} x^\pi_{i k} \otimes u^{F(\pi)}_{k j},&
\omega_{12}(x^\pi_{i j}) &= \delta_{1,\pi}.
\end{align*}
Here, the $u^\pi_{ij}$ are the matrix coefficients of the representation $\pi$.
As the $\delta_i$ are $\ast$-homomorphisms, the algebra structure is also determined by this.

Exchanging the role of $G_1$ and $G_2$, we get another linking algebra $\cO(G_2,G_1)$ and invariant state $\omega_{21}$.
There is a canonical isomorphism $\cO(G_2,G_1) \cong \cO(G_1,G_2)^{\opo}$ given by $y_{ji}^\pi \mapsto (x_{ij}^\pi)^*$, where the matrix coefficients $y^\pi_{j i}$ are defined analogously to $x^\pi_{i j}$. Moreover, there exists a unital $\ast$-homomorphism $\theta_1 \colon \cO(G_1) \to \cO(G_1,G_2) \otimes \cO(G_2,G_1)$ defined on matrix elements of unitary representations $u_{ij}^\pi \in \cO(G_1)$ by  
\[
\theta_1(u^\pi_{ij}) = \sum_{k = 1}^{\dim(F(\pi))} x_{ik}^\pi \otimes y_{kj}^\pi.
\]
Note that $\theta_1$ is state preserving, in the sense that 
\[
h_1(x)1 = (\omega_{12} \otimes \id) \theta_1(x) = (\id \otimes \omega_{21}) \theta_1(x) \qquad (x \in \cO(G_1)).
\]

Let $G$ be a compact quantum group and denote the dual $\ast$-algebra of $\cO (G)$ by $\cU(G)$ (the $\ast$-structure is  defined $\omega^*(a) = \omega(S(a)^*)$ for $\omega \in \mathcal U(G)$ and $a \in \cO(G)$).  We write  $\cU(G \times G)$ for the dual of $\cO(G) \otimes \cO(G)$. The ``coproduct'' map $\hat \Delta: \cU(G) \to \cU(G \times G)$ is defined to be dual of
the product map $\cO(G) \otimes \cO(G) \to \cO(G)$.  See \cite{MR3204665} for more details.  In the following, we are mostly interested in monoidal equivalences given by unitary dual $2$-cocycles in the following sense.

\begin{definition}
A \emph{unitary dual $2$-cocycle} for $G$ is given by a unitary element $\sigma \in \cU(G \times G)$ such that
\[
\sigma_{1, 2} \sigma_{1 2, 3} = \sigma_{2, 3} \sigma_{1, 2 3}.
\]
\end{definition}

Here, $\sigma_{1 2, 3}$ denotes $(\hat\Delta \otimes \id)(\sigma) \in \cU(G \times G \times G)$.
Without losing generality we can assume the normalization condition
\[
(\epsilon \otimes \id)(\sigma) = 1 = (\id \otimes \epsilon)(\sigma)
\]
for the trivial representation $\epsilon \colon \cU(G) \to \bC$, which we always do.
Alternatively, we can interpret $\sigma$ as a bilinear form $\cO(G) \times \cO(G) \to \bC$ satisfying
\begin{align*}
\sigma(a_{(1)}, a'_{(1)}) \sigma(a_{(2)} a'_{(2)}, a'') &= \sigma(a'_{(1)}, a''_{(1)}) \sigma(a, a'_{(2)} a''_{(2)}),&
\sigma(1, a) = \epsilon(a) = \sigma(a, 1),
\end{align*}
and unitarity for the convolution algebra structure.

Given such $\sigma$ and a left $\cO(G)$-comodule $\ast$-algebra $B$, we can twist the product of $B$ to a new associative product
\[
b^1 \tensor[_\sigma]{\cdot}{} b^2 = \sigma(b^1_{(1)}, b^2_{(1)}) b^1_{(2)} b^2_{(2)},
\]
where we denote the coaction as
\[
B \to \cO(G) \otimes B, \quad b \mapsto b_{(1)} \otimes b_{(2)}.
\]
With the new involution (see, e.g.,~\cite{MR3194750}) given by
\[
b^\sharp = \overline{\sigma(S^{-1}(b_{(2)}), b_{(1)})} b_{(3)}^* = \sigma^*(b_{(2)}^*, S^{-1}(b_{(1)}^*)) b_{(3)}^*,
\]
we obtain a $\ast$-algebra $(B, \tensor[_\sigma]{\cdot}{}, \sharp)$, which we denote by $\tensor[_\sigma]{B}{}$.
This cocycle deformation is compatible with C$^\ast$-structures: if $B$ is a C$^\ast$-algebra and $B \to B \otimes C^u(G)$ is a C$^\ast$-algebraic coaction and $\sigma$ is as above, then a similar construction can be carried through, resulting in a C$^\ast$-algebra $_\sigma B$.  See \cite{NT-actions} for details.  In the particular case of a finite-dimensional C$^\ast$-algebra $B$ of interest to us, the resulting algebraic construction and C$^\ast$-algebraic construction coincide.  
We also note that if $\psi$ is a $G$-invariant state on $B$, then $\psi$ remains a state when viewed as a functional on $_\sigma B$.   

Similarly, when $B$ has a right coaction of $\cO(G)$, we can define a new associative product by
\[
f^1 \cdot_{\sigma^{-1}} f^2 =  \sigma^{-1}(f^1_{(2)}, f^2_{(2)}) f^1_{(1)} f^2_{(1)},
\]
and a compatible $\ast$-structure.
We denote this $\ast$-algebra by $B_{\sigma^{-1}}$.

\begin{definition} \label{CQG-twist} [cf.~\cite{MR1213985}]
With $G$ and $\sigma$ as above, we denote by $\cO(G)^\sigma$ the Hopf $\ast$-algebra with the underlying coalgebra $\cO(G)$ and algebra $\tensor[_\sigma]{\cO(G)}{_{\sigma^{-1}}}$.
We also write $G^\sigma$ for the compact quantum group represented by $\cO(G)^\sigma$.
\end{definition}

Recall that $G^{\sigma}$ can be directly characterized in terms of the structure of $\cU(G^\sigma)$: as an algebra it is the same as $\cU(G)$, but the coproduct is given by $\Delta_\sigma(T) = \sigma \Delta(T) \sigma^{-1}$.
There is a unitary monoidal equivalence $F\colon \Rep(G) \to \Rep(G^\sigma)$, whose underlying C$^\ast$-functor is the identity functor and the tensor transform $F(\pi) \otimes F(\pi') \to F(\pi \otimes \pi')$ is given by the action of $(\pi \otimes \pi')(\sigma^{-1})$ on $H_\pi \otimes H_{\pi'}$.

\begin{remark}\label{rmk-dim-pres}
Up to isomorphism, any unitary monoidal equivalence $F \colon \Rep(G) \to \Rep(G')$ satisfying $\dim H_\pi = \dim H_{F(\pi)}$ is of this form.   See \cite[Section 4]{MR2202309}
\end{remark}

The linking algebra $\cO(G^\sigma, G)$ is given by $\tensor[_\sigma]{\cO(G)}{}$.
More generally, if $B$ is any unital $\ast$-algebra and $B \to \cO(G) \otimes B$ is a coaction, then this same linear map defines a coaction $\tensor[_\sigma]{B}{} \to \cO(G^\sigma) \otimes \tensor[_\sigma]{B}{}$.  See \cite[Proposition 4.11]{MR2202309} and \cite{MR1213985}.  

\begin{remark}\label{rem:comparison-of-twists}
When $B$ is a finite dimensional C$^*$-algebra endowed with a left $\cO(G)$-comodule structure and a $G$-invariant state $\psi$, the twisting $\tensor[_\sigma]{B}{}$ comes from the above monoidal equivalence $F$ up to an isomorphism.
To be more precise, let $R$ be the unitary antipode on $\cO(G)$.
Then the opposite algebra $B^{\opo} = \{b^\opo \mid b \in B\}$ admits a right $\cO(G)$-comodule $*$-algebra structure given by $b^\opo \mapsto b_{(2)}^\opo \otimes R(b_{(1)})$.
Together with the GNS inner product for the invariant state $\psi'(b^\opo) = \psi(b)$, we get an object of $\Rep(G)$ represented by $B^\opo$.
The corresponding $\Rep(G^\sigma)$-algebra has the product
\[
b^{1\opo} \tensor[_\sigma]{\cdot}{} b^{2\opo} = \sigma^{-1}(R(b^1_{(1)}), R(b^2_{(2)})) (b^{2\opo}_{(2)} b^{1\opo}_{(2)})^{\opo} \quad (b^1, b^2 \in B).
\]
In other words, the corresponding left $\cO(G)$-comodule algebra is $B$ with the twisted product
\[
b^1 \tensor[_\sigma]{\cdot}{^\prime} b^2 = \sigma^{-1}(R(b^2_{(1)}), R(b^1_{(2)})) b^1_{(2)} b^2_{(2)} = (\hat R \otimes \hat R)(\sigma^{-1})_{2 1} (b^1_{(1)}, b^2_{(1)}) b^1_{(2)} b^2_{(2)}
\]
for the unitary antipode $\hat R$ of $\cU(G)$.
By \cite{MR2806547}*{Proposition 5.3}, $\sigma$ and $(\hat R \otimes \hat R)(\sigma^{-1})_{2 1}$ are cohomologous, hence the C$^*$-algebra $(B, \tensor[_\sigma]{\cdot}{^\prime})$ is isomorphic to $\tensor[_\sigma]{B}{}$.
\end{remark}

Turning to the quantum automorphism groups, they obey the same fusion rules as $\SO(3) = \Aut(M_2)$.
In fact, these compact quantum groups have the following rigidity properties.

\begin{theorem}[\citelist{\cite{MR1709109}\cite{MR3240820}}]\label{thm:banica-mrozinski}
The compact quantum groups $\Aut^+(B, \psi)$ for a finite-dimensional C$^\ast$-algebra $B$ and a $\delta$-form $\psi$ have the same fusion rules as $\SO(3)$.
Conversely, when $G$ is a compact quantum group with the same fusion rules as $\SO(3)$, there is such $B$ and $\psi$ satisfying $G \cong \Aut^+(B, \psi)$.
\end{theorem}

Note that if $G$ compact quantum group that has the fusion rules of $SO(3)$, $G = \Aut^+(B, \psi)$ where $B$ is, as a representation of $G$, represented by the direct sum of the trivial representation and the one corresponding to the irreducible $3$-dimensional representation of $SO(3)$. 

\begin{theorem}[\cite{DeVa10}]\label{thm:derijdt-vandervennet}
Let $(B_i, \psi_i)$ be finite-dimensional C$^\ast$-algebras with $\delta_i$-forms $\psi_i$, for $i=1,2$.
Then the compact quantum groups $\Aut^+(B_i, \psi_i)$ are monoidally equivalent if and only if $\delta_1 = \delta_2$.
\end{theorem}

In this case, the linking algebra can be characterized as the universal $\ast$-algebra generated by the coefficients of a unital $\ast$-homomorphism 
\[
\rho\colon B_2 \to B_1 \otimes \cO(\Aut^+(B_1,\psi_1), \Aut^+(B_2,\psi_2))
\]
satisfying the $\psi_2$-$\psi_1$-invariance condition
\[
(\psi_1\otimes 1)\rho(x) = \psi_2(x)1 \qquad (x \in B_2).
\]
In fact, the nontriviality of this universal algebra characterizes the existence of monoidal equivalence between $\Aut^+(B_i, \psi_i)$.
Note that if $\psi_i$ are the respective Plancherel traces, we have $\delta_i = \sqrt {\dim B_i}$, and therefore the corresponding quantum groups $\Aut^+(B_i)$ (of interest to us in this paper) are monoidally equivalent if and only if $\dim B_1 = \dim B_2$.

\section{Isomorphisms from \texorpdfstring{$2$}{2}-cocycle Deformations}
\label{sec:IsoFromCocycleDeformations}

In this section we establish our first main result of the paper--that every quantum automorphism group $\Aut^+(B)$ arises as a cocycle twist of a quantum permutation group $S_N^+$.
The special case when $B = M_n$ was established in \cite{BaBiCu11}.
The general case is a straightforward generalization of the arguments there.
We then apply this structure result to establish our first crossed product isomorphisms. 

Given natural numbers $n_1, \dots, n_m$, put $Y_r = \{0 \le i < n_r\}$, $X_r = Y_r \times Y_r$, and $X = X_1 \coprod \dots \coprod X_m$.
We also write $N = \abs{X} = \sum_r n_r^2$.
Note that each $X_r$ has a free transitive action of $\Gamma_r = \bZ_{n_r} \times \bZ_{n_r}$.
This induces an embedding
\[
\Gamma = \Gamma_1 \times \dots \times \Gamma_m < \Sym(X) \cong S_N.
\]
Recall that $\hat \Gamma_t$ has a $\bT$-valued nondegenerate $2$-cocycle:
up to a choice of identification $\hat \Gamma_t \cong \Gamma_t$ and difference by coboundary, it can be written as
\[
\omega'_t([j_1, j_2], [k_1, k_2]) = \exp\mathopen{}\left(\frac{2 \pi i}{n_t} j_1 k_2\right).
\]
Let $\omega_t$ be a $2$-cocycle cohomologous to $\omega'_t$ satisfying $\omega_t(h, h^{-1}) = 1$.
Concretely, we can choose $\psi(h) \in \bT$ such that $\psi(h)^{-2} = \omega'_t(h, h^{-1})$, and take $\omega_t = \omega'_t \partial \psi$.

Then the product $\omega_1 \times \dots \times \omega_m$ is a (nondegenerate) $2$-cocycle on $\hat \Gamma$, which can be presented as a convolution invertible map $\sigma_0 \colon \cO(\Gamma) \times \cO(\Gamma) \to \bC$, with the additional normalization
\begin{equation}\label{eq:cocycle-normaliszation-2}
\sigma_0(g, g^{-1}) = 1 \quad (g \in \Gamma).
\end{equation}

Composing the restriction maps $\cO(S_{N}^+) \to \cO(S_N) \to \cO(\Gamma)$, we obtain a convolution invertible map
\[
\sigma \colon \cO(S_N^+) \otimes \cO(S_N^+) \to \bC.
\]
This shall be the cocycle of interest in the sequel.

\begin{theorem}
\label{thm:cocycle-twist-rel}
We have $\cO(S_N^+)^\sigma \cong \cO(\Aut^+(\bigoplus_{i=1}^m M_{n_r}))$.
\end{theorem}

\begin{proof}
The corepresentation category of $\cO(S_N^+)^\sigma$ is monoidally equivalent to $\Rep S_N^+$, and the general classification result of such compact quantum groups, discussed in the previous section, implies
\[
\cO(S_N^+)^\sigma \cong \cO(\Aut^+(B, \psi))
\]
for some finite-dimensional C$^\ast$-algebra $B$ and a $\delta$-form $\psi$, such that $\delta^2 = N$.
Moreover, the cocycle twisting does not change the vector spaces underlying the representations, hence we have $B = C(X)$ as a vector space.
We thus have $\delta^2 = N =  \dim B$, hence $\psi$ must be the Plancherel trace on $B$.

By Remark \ref{rem:comparison-of-twists}, the product map of $B$ is given by the composition of the action of $\sigma$ and the product map of $C(X)$, which is the product map of ${}_\sigma C(X)$.  That is,  $B = \tensor[_{\sigma}]{C(X)}{}$.
Now, by the product structure of $\sigma_0$, each block $X_r$ of $X$ gives a copy of $M_{n_r}$, while different blocks remain orthogonal. This shows that
\[
B \cong \bigoplus_{r=1}^m M_{n_r},
\]
and we obtain the claim.
\end{proof}

\begin{remark}\label{rem:from-referee}
It was kindly pointed out to us by an anonymous referee that the above theorem actually admits quite a broad generalization, as follows. We thank the referee for allowing us to include this result and for conveying the idea of the proof. 
\begin{theorem}
Let $(B,\psi)$ be finite dimensional C$^\ast$-algebra equipped with a $\delta$-form $\psi$, and let $\sigma$ be a unitary $2$-cocycle for $\Aut^+(B,\psi)$.  Then $\Aut^+(B,\psi)^\sigma \cong \Aut^+(_{\sigma} B, \psi)$ canonically. 
\end{theorem}

\begin{proof}
Let $G = \Aut^+(B,\psi)$, and form $_\sigma B$ and $G^\sigma$ as described in Section \ref{sec:Preliminaries}.  Recall that $_\sigma B = B$ and $\cO(G^\sigma) = \cO(G)$ as comodules and coalgebras respectively.  Moreover, the map  $_\sigma B \to {}_\sigma B \otimes \cO(G^\sigma)$ is a $\psi$-preserving coaction.    Hence we obtain a surjective morphism of Hopf $\ast$-algebras $\pi:\cO(\Aut^+(_\sigma B, \psi)) \to \cO(G^\sigma)$.  That is, $G^\sigma < \Aut^+(_\sigma B, \psi)$ as a quantum subgroup. But since these two quantum groups have the same fusion rules, it follows that $\ker \pi = \{0\}$ and thus $G^\sigma = \Aut^+(_\sigma B, \psi)$.    
\end{proof}

\end{remark}

\begin{corollary}\label{cor:cross-prod-isoms}
In the above setting of Theorem \ref{thm:cocycle-twist-rel}, with $B = \bigoplus_{r=1}^m M_{n_r}$, there are $h$-preserving actions of $\Gamma^2$ on $\cO(S_N^+)$ and $\cO(\Aut^+(B))$ such that
\[
C^r(S_N^+) \rtimes \Gamma^2 \cong C^r(\Aut^+(B)) \rtimes \Gamma^2,
\]
intertwining the induced traces.
\end{corollary}

\begin{proof}
This is a consequence of the structure theory for the general $2$-cocycle deformation scheme (cf.~\cite{arXiv:1107.2512}*{Proposition 9}), but let us elaborate on this for the reader's convenience.

Let $\Lambda$ be a finite commutative group, and $\hat \Lambda$ its Pontryagin dual.
An action of $\Lambda$ on a C$^\ast$-algebra $A$ is the same thing as a grading $A = \bigoplus_{k \in \hat \Lambda} A_k$ by $\hat \Lambda$, or a coaction $\alpha \colon A \to C^*(\hat \Lambda) \otimes A$.
For $k \in \hat \Lambda$, let us write $\alpha^{(k)}(x)$ for the projection of $x$ to $A_k$.

Now, let $\sigma$ be a $2$-cocycle on $\hat \Lambda$, normalized as in \eqref{eq:cocycle-normaliszation-2}.
The deformed algebra $\tensor[_\sigma]{A}{}$ is still graded by $\hat \Lambda$.
Let us denote the corresponding coaction map $\tensor[_\sigma]{A}{} \to C^*(\hat \Lambda) \otimes \tensor[_\sigma]{A}{}$ by $\alpha_\sigma$.

We claim that $A \rtimes \Lambda$ is isomorphic to $\tensor[_\sigma]{A}{}  \rtimes \Lambda$.
The corollary follows from this by considering the action of $\Lambda = \Gamma^2$ on $\cO(S_N^+)$ obtained by combining left and right translations.

We identify $\tensor[_\sigma]{A}{}$ with the algebra $A''$ generated by $\lambda^{(\sigma)}_k \otimes a$ on $\ell^2(\hat\Lambda) \otimes A$ for $k \in \hat \Lambda$ and $a \in A_k$, where $\lambda^{(\sigma)}_k$ is the regular $\sigma$-representation $\lambda^{(\sigma)}_k \delta_{k'} = \sigma(k, k') \delta_{k k'}$.  (Here $(\delta_k)_{k \in \hat\Lambda}$ is an orthonormal basis for $\ell^2(\hat \Lambda)$).
Then the coaction $\alpha_\sigma$ becomes $\alpha'_\sigma\colon\lambda^{(\sigma)}_k \otimes a \mapsto \lambda_k \otimes \lambda^{(\sigma)}_k \otimes a$ on $A''$.
Thus, the crossed product
\[
C(\hat \Lambda) \ltimes A'' \cong \tensor[_\sigma]{A}{} \rtimes \Lambda
\]
is represented by the C$^\ast$-algebra on the right Hilbert $A$-module $\ell^2(\hat \Lambda)^{\otimes 2} \otimes A$ generated by $(\chi_g)_1$ for $g \in \hat \Lambda$ and $\sum_k \lambda_k \otimes \lambda^{(\sigma)}_k \otimes \alpha^{(k)}(x)$ for $k \in \hat \Lambda$ and $x \in A$.  (Here $\chi_g \in C(\hat \Lambda)$ denotes the characteristic function of $\{g\}$.)

Similarly, we obtain an algebra $A' \cong A$ instead of $A''$ by removing $\sigma$, and an analogous spatial presentation of $A \rtimes \Lambda \cong C(\hat \Lambda) \ltimes A'$.
Let $V$ be the unitary operator $\delta_k \otimes \delta_{k'} \mapsto \sigma(k^{-1}, k') \delta_k \otimes \delta_{k'}$ on $\ell^2(\hat\Lambda)^{\otimes 2}$.
We show that $\Phi = \Ad_{V_{1 2}}$ conjugates $C(\hat \Lambda) \ltimes A''$ to $C(\hat \Lambda) \ltimes A'$.

If $k \in \hat \Lambda$ and $x \in A''$, the action of $\Phi(\alpha'_\sigma(x) (\chi_g)_1)$ on the vector $\delta_k \otimes \delta_{k'} \otimes b$ is given by
\[
\sum_h \delta_{g, k} \overline{\sigma(k^{-1}, k')} \sigma(h, k') \sigma(k^{-1} h^{-1}, h k') \delta_{h k} \otimes \delta_{h k'} \otimes \alpha^{(h)}(x) b.
\]
Using the cocycle identity and \eqref{eq:cocycle-normaliszation-2} for $\sigma$, we see that this is equal to
\[
\sum_h \sigma(h, g) \delta_{g, k} \delta_{h k} \otimes \delta_{h k'} \otimes \alpha^{(h)}(x) b.
\]
This is equal to the action of
\[
\sum_{h} \sigma(h, g) (\lambda_h \otimes \lambda_h \otimes \alpha^{(h)}(x)) (\chi_g)_1 = \sum_h \sigma(h, g) \alpha'(\alpha^{(h)}(x)) (\chi_g)_1,
\]
which is indeed in $C(\hat \Lambda) \ltimes A'$.
\end{proof}

\begin{remark}\label{rem:FIELA}
The above corollary also holds at the purely algebraic level, as
\[
\cO(S_N^+) \rtimes \Gamma^2 \cong \cO(\Aut^+(B)) \rtimes \Gamma^2.
\]
This can be seen either by analyzing the proof, or by considering the induced coactions of the coalgebras $\cO(S_N^+) \cong \cO(\Aut^+(B))$.
Moreover, we also get isomorphisms
\begin{align*}
    \cO(S_N^+) \rtimes \Gamma^2 \cong \cO(\Aut^+(B),S_N^+) \rtimes \Gamma^2 \cong \cO(S_N^+,\Aut^+(B)) \rtimes \Gamma^2 ,
\end{align*}
since $\cO(\Aut^+(B),S_N^+) = \tensor[_\sigma]{\cO(S_N^+)}{}$ and similarly for the opposite linking algebra.
\end{remark}

By the Takesaki--Takai duality \cite{Takai}, we obtain the following.

\begin{corollary}\label{cor:dbl-cross-prod-is-mat-amp}
In the setting of Corollary \ref{cor:cross-prod-isoms}, there is a trace-preserving action of $\Gamma^2$ on $C^r(S_N^+) \rtimes \Gamma^2$ such that
\[
(C^r(S_N^+) \rtimes \Gamma^2) \rtimes \Gamma^2 \cong C^r(\Aut^+(B)) \otimes M_{d^4},
\]
where $d = \prod_r n_r = \sqrt{\abs{\Gamma}}$.
Moreover, this isomorphism intertwines the natural traces on both sides.
\end{corollary}

Another consequence of the crossed product isomorphism in Corollary \ref{cor:cross-prod-isoms} are the embeddings (see \citelist{\cite{MR2492990}*{Lemma 4.1}})
\begin{equation}\label{eq:O-Aut-B-emb-to-ampl-S-N-plus}
\cO(\Aut^+(B)) \hookrightarrow M_{d^2} \otimes \cO(S_N^+) \otimes M_{d^2} \cong \cO(S_N^+) \otimes M_{d^4} \quad \& \quad \cO(S_N^+) \hookrightarrow \cO(\Aut^+(B)) \otimes M_{d^4}.
\end{equation}
and various parallels for the C$^*$-algebraic and von Neumann algebraic settings, which has further implications for the structure of the associated algebras $\Aut^+(B)$ as we will see in the next section.

In Section \ref{sec:unitary-error-basis}, we improve on the above and give an isomorphism between the iterated crossed product by $\Gamma$ and amplification by $M_{d^2}$.

\section{Monoidal Equivalence and Matrix Models}
\label{sec:mon-equiv-mat-model}

We now consider applications of the above crossed product isomorphism results to the study of various external approximation properties for the associated quantum group algebras, including the Connes Embedding property, residual finite-dimensionality, and inner unitarity. We begin by working in a slightly more general framework and consider how these approximation properties can be transferred between monoidally equivalent quantum groups.

\begin{definition}[\cite{BrChFr18}]
A unital $\ast$-algebra $\cA$ is said to be \emph{residually finite-dimensional}, or \emph{RFD}, if there is an injective $\ast$-homomorphism
\[
\cA \to \prod_{i \in I} M_{d_i}
\]
for some index set $I$ and a family of positive integers $(d_i)_{i \in I}$.
\end{definition}

We consider this property for the $\ast$-algebras of the form $\cO(G)$ for some compact quantum group.
When $\cO(G)$ is RFD, then its Haar state $h$ is an amenable trace and the quantum group von Neumann algebra $L^\infty(G)$ has the Connes Embedding Property (CEP) \cite{BhBrChWa17}.  Recall that a finite von Neumann algebra $(\cM, \tau)$ has the CEP if and only if there exists a $\tau$-preserving embedding $\cM \hookrightarrow \mathcal R^\omega$, where $\mathcal R^\omega$ is an ultrapower of the hyperfinite II$_1$-factor $\mathcal R$.

\subsection{Transfer of finite approximations}

The following theorem shows that under the assumption of the existence of a non-zero finite-dimensional $\ast$-representation of the linking algebra  $\cO(G_1,G_2)$ on a Hilbert space, then one can transfer many approximate finitary representation-theoretic properties from one quantum group to another.

\begin{theorem}\label{T: transferring matrix models RFD}
Let $G_1$ and $G_2$ be monoidally equivalent compact quantum groups, and assume that there exists a non-zero finite-dimensional $\ast$-representation $\pi \colon \cO(G_1,G_2) \rightarrow M_d$ of the linking algebra.
Then the following statements are equivalent.
\begin{enumerate}
    \item $\cO(G_2)$ is RFD.
    \item $\cO(G_1,G_2)$ is RFD.
    \item $\cO(G_1)$ is RFD.
\end{enumerate}
\end{theorem}

\begin{proof}
By symmetry, it is enough to prove the implications $(1) \implies (2)$ and $(2) \implies (3)$.

$(1) \implies (2)$.  Define a $\ast$-homomorphism $\rho\colon\cO(G_1,G_2) \to M_d \otimes \cO(G_2)$ by $\rho = (\pi \otimes \id)\delta_2$.  Note that $\rho$ is injective because
\[
\rho(x) = 0 \implies \rho(x^*x) = 0 \implies 0 = (\id \otimes h_2)\rho(x^*x) = \pi((\id  \otimes h_2)\delta_2(x^*x)) = \omega_{12}(x^*x) \pi(1) , 
\]
and $\omega_{12}$ is faithful.
Since $M_d \otimes \cO(G_2)$ is RFD and $\rho$ is injective, we are done.

$(2) \implies (3)$. Consider the injective $\ast$-homomorphism $\theta_1 \colon \cO(G_1) \to \cO(G_1,G_2) \otimes \cO(G_2,G_1)$ (see Section \ref{sec:mon-equiv-mat-model}). Since $\cO(G_1,G_2)$ is RFD if and only if the opposite algebra $\cO(G_2,G_1)$ is RFD, and tensor products of RFD algebras are RFD (indeed -- if $\cA \hookrightarrow \prod_i M_{d_i}$, $\cB \hookrightarrow \prod_j M_{d_j}$, then $\cA \otimes \cB \hookrightarrow \prod_{i,j} M_{d_id_j}$), we conclude that $\cO(G_1)$ is RFD.  
\end{proof}

Since in the the proof of Theorem \ref{T: transferring matrix models RFD} we established the existence of various state-preserving embeddings, the preceding theorem admits a natural analogue for the CEP.

\begin{theorem} \label{T: transferring matrix models CEP}
Let $G_1$ and $G_2$ be monoidally equivalent compact quantum groups of Kac type.
Assume that there exists a non-zero finite-dimensional $\ast$-representation of the linking algebra $\cO(G_1,G_2)$.
Then the following statements are equivalent.
\begin{enumerate}
    \item $L^\infty(G_2)$ has the CEP.
    \item $L^\infty(G_1,G_2)$ has the CEP.
    \item $L^\infty(G_1)$ has the CEP.
\end{enumerate}
\end{theorem}

\begin{proof}
First note that the invariant state $\omega_{12}$ on $\cO(G_1,G_2)$ is a trace, since both $G_i$ are of Kac type.
Next, we note that in the proof of Theorem \ref{T: transferring matrix models RFD}, we have constructed trace-preserving embeddings
\begin{gather*}
\begin{aligned}
\rho\colon L^\infty(G_1,G_2) &\to M_d \otimes L^\infty(G_2),&
\theta_1\colon L^\infty(G_1) &\to L^\infty(G_1,G_2) \mathbin{\bar\otimes} L^\infty(G_2,G_1),
\end{aligned}\\
\sigma\colon L^\infty(G_2) \to M_d \otimes L^\infty(G_1) \otimes M_d.
\end{gather*}
From these embeddings, it follows that (1)---(3) are equivalent.
\end{proof}

\begin{remark}
The existence of a finite-dimensional representation $\pi\colon \cO(G_1,G_2) \to M_d$ is a fairly restrictive assumption.
For example, it forces the unitary fiber functor $F\colon \Rep(G_1) \to \Rep(G_2)$ to be dimension preserving, i.e., we have $\dim H_\rho = \dim H_{F(\rho)}$ for each $\rho \in \Rep(G_1)$.
This follows from the fact that $(\id \otimes \pi)(X^\rho) \in M_{d(\rho), d(F(\rho))} \otimes M_d$ is a finite-dimensional unitary operator, which can happen only if $\dim H_\rho = \dim H_{F(\rho)}$.
In particular, this forces $G_2$ to be a $2$-cocycle deformation of $G_1$ (cf. Remark \ref{rmk-dim-pres}).

On the other hand, this assumption on the existence of $\pi$ is necessary for the conclusions to hold.  For example, if $n \ge 3$ and $q \in (-1,0)$ is such that $q+q^{-1} = -n$, then $O_n^+ \sim_m SU_q(2)$ \cite{MR2202309}, $\cO(O_n^+)$ is RFD \cite{chi-rfd}, while on the other hand $\cO(SU_q(2))$ cannot be RFD because it is not of Kac type \cite{So05}.
\end{remark}

\begin{proposition}
\label{prop:fin-dim-rep-link-Autplus-SNplus}
Let $B$ be a finite-dimensional C$^\ast$-algebra, and put $N = \dim B$.
Then the linking algebra $\cO(\Aut^+(B), S_N^+)$ has a finite-dimensional representation.
\end{proposition}

\begin{proof}
By Theorem \ref{thm:cocycle-twist-rel}, we have $\Aut^+(B) \cong (S_N^+)^\sigma$ for some dual $2$-cocycle $\sigma$ induced from some such $\sigma_0$ for a finite subgroup $G < S_N^+$.
Thus the linking algebra $\cO(\Aut^+(B), S_N^+) = \tensor[_\sigma]{\cO(S_N^+)}{}$ has a finite-dimensional quotient $\tensor[_\sigma]{\cO(G)}{}$.
\end{proof}

\begin{corollary}\label{cor:Aut-RFD-and-CEP}
Let $B$ be a finite-dimensional C$^\ast$-algebra equipped with any faithful trace $\psi$.  Then $\cO(\Aut^+(B,\psi))$ is RFD and hence has the CEP.
\end{corollary}

\begin{proof}
The claim for $S_N^+$ is established in \cite{BrChFr18}.
This, together with Theorem \ref{T: transferring matrix models RFD} and Proposition \ref{prop:fin-dim-rep-link-Autplus-SNplus}, implies the claim when $\psi$ is the Plancherel trace.
The general case then follows from the free product decomposition in Proposition \ref{P: general trace} and the stability of these properties with respect to fee products \cite{EL}. 
 \end{proof}

\subsection{Inner faithful representations}
\label{sec:inn-faith-free-aut}

\begin{definition}[\cites{bb-inner,MR2681256}]
Let $G$ be a compact quantum group and $\mathcal A$ be unital $\ast$-algebra.
A $\ast$-homomorphism $\pi\colon\cO(G) \to \mathcal A$ is said to be {\it inner faithful} if $\ker (\pi)$ does not contain any non-zero Hopf $\ast$-ideal.
If there exists a finite-dimensional C$^\ast$-algebra $\mathcal A$ and an inner faithful $\pi\colon \cO(G) \to \mathcal{A}$, then we call $\cO(G)$ an {\it inner unitary} Hopf $\ast$-algebra.
\end{definition}

The notion of inner unitarity for Hopf $\ast$-algebras is a quantum analogue of (a strong form of) linearity for discrete groups. 
Recall that a discrete group $\Gamma$ is called a {\it linear group} if there exists a faithful group homomorphism $\pi\colon \Gamma \to \mathrm{GL}_d(\bC)$. In this context, we have that $\Gamma$ is linear if and only if the group algebra $\bC \Gamma$ admits an inner faithful homomorphism to $M_d$ (without assuming compatibility for $\ast$-structures) \cites{bb-inner}.  If the morphism $\pi$ is $\ast$-preserving, this is equivalent to saying that we have an embedding $\Gamma \hookrightarrow U_d$.
In  general, if $\cO(G)$ is inner unitary, then it is RFD \cite{BrChFr18}.

\begin{proposition}\label{prop:Aut-Inner-Unitary}
Let $B$ be a finite dimensional C$^\ast$-algebra such that dim$(B)$ lies outside the range $[6,9]$.
Then $\cO(\Aut^+(B))$ is inner unitary.
\end{proposition}

\begin{proof}
On the one hand, the Hopf $\ast$-algebras $\cO(S_N^+)$ are inner unitary for all $N$ outside the range $[6,9]$ by \cite{BrChFr18}*{Theorem 4.11}.
On the other, we have an embedding of the form \eqref{eq:O-Aut-B-emb-to-ampl-S-N-plus}.
Then a straightforward adaptation of \cite{bb-inner}*{Theorem 6.3} to the unitary setting implies that $\cO(\Aut^+(B))$ must then also be inner unitary.
\end{proof}

Now, taking $B = M_n$ and using the canonical embedding $\cO(\Aut^+(B)) \subseteq \cO(O_n^+)$, we are able to lift inner unitarity to free orthogonal quantum  groups.

\begin{theorem}\label{thm:FO-Inner-Unitary}
The Hopf $\ast$-algebra $\cO(O_n^+)$ is inner unitary for $n=2$ and $n \geq 4$.
\end{theorem}

\begin{proof}
We denote the conditional expectation $\cO(O_n^+) \to \cO(\Aut^+(M_n))$ preserving the Haar state by $E$.
Let $\pi \colon \cO(\Aut^+(M_n)) \to B(H)$ be a $\ast$-homomorphism for some finite-dimensional Hilbert space $H$ that gives the inner unitarity of $\cO(\Aut^+(M_n))$.
Consider the Hilbert space $\tilde H = \cO(O_n^+) \otimes_{\cO(\Aut^+(M_n))} H$, where, as usual, the inner product is given by
\[
(f' \otimes \xi', f \otimes \xi) = (\pi(E(f^* f')) \xi', \xi).
\]
The left multiplication defines a $\ast$-homomorphism $\tilde\pi\colon \cO(O_n^+) \to B(\tilde H)$.

Since $\cO(O_n^+)$ is finitely generated as a right $\cO(\Aut^+(M_n))$-module (we can take $1, u_{i j}$ for $1 \le i, j \le n$ as generators), $\tilde H$ is finite-dimensional.
We claim that $\ker \tilde \pi$ does not contain any nonzero Hopf $\ast$-ideal.

Let $I$ be a Hopf $\ast$-ideal of $\cO(O_n^+)$ contained in $\ker \tilde \pi$, and put $\cO(G) = \cO(O_n^+) / I$.
To show that $I = 0$, it is enough to show that the restriction functor $F\colon \Rep O_n^+ \to \Rep G$ is full.
By Frobenius reciprocity, this is equivalent to the claim that $F(U^{\otimes 2})$ contains the unit with multiplicity one whenever $U$ is a nontrivial irreducible object of $\Rep O_n^+$.

Take a nontrivial irreducible object $U$ from $\Rep O_n^+$.  First suppose that $U \in \Rep \Aut^+(M_n)$.  Then, as  $\tilde \pi|_{\cO(\Aut^+(M_n))}$ contains $\pi$ as a direct summand, we have $I \cap \cO(\Aut^+(M_n)) = 0$.
This implies that $F(U) \in \Rep G$ is a nontrivial irreducible object as well. For the general case of $U$ being an irreducible object from $\Rep O_n^+$, note that we have $U^{\otimes 2} \in \Rep \Aut^+(M_n)$, hence from the above we have that $F(U^{\otimes 2})$ has the same irreducible decomposition as $U^{\otimes 2}$.
But this implies that $F(U^{\otimes 2})$ contains the unit with multiplicity one.
\end{proof}

\begin{remark}
The first half of the above proof is a special case of the one for \cite{MR2681256}*{Theorem 5.7}.
Indeed, the following generalization holds: if $\cO(G)$ is the regular algebra of a compact quantum group $G$, and $A = \cO(G')$ is a inner unitary Hopf $\ast$-subalgebra of $\cO(G)$ closed under the adjoint action, and such that $\cO(G)$ is finitely generated over $A$, then $\cO(G)$ is also inner unitary.
To see this, one can observe that the proof of \cite{MR2681256}*{Theorem 5.7} as its commutativity assumption on $A$ was only to make sure that $\cO(G)$ is faithfully flat over $A$, which holds in the above setting \cite{MR3263140}.
\end{remark}

\section{Strong \texorpdfstring{$1$}{1}-boundedness of quantum automorphism group factors}
\label{sec:S1B}

Let $(\mathcal M,\tau)$ be a tracial von Neumann algebra, and let $X = (X_1,\dots,X_n)$ be an $n$-tuple of self-adjoint elements of $\cM$ and $Y$ another such $m$-tuple.
Recall Voiculescu's relative microstates free entropy $\chi(X:Y)$ introduced in \cite{MR1371236}, which we from now on will refer to as just \emph{relative free entropy}.
Using this, Voiculescu defined the \emph{(modified) free entropy dimension}
\begin{align*}
    \delta_0(X) = n + \lim_{\varepsilon \downarrow 0} \frac{\chi(X + \varepsilon S : S)}{\abs{\log(\varepsilon)}} ,
\end{align*}
where $S = (S_1,\dots,S_n)$ is a free family of semicircular elements free from $X$.
It satisfies $\delta_0(X) \leq n$ and equality is attained for instance when $X$ consists of $n$ free semicircular elements.
Thus a free group factor $\mathcal{L}\mathbb{F}_m$ admits a generating set with the property that its free entropy dimension is precisely $m$.
While it is unknown whether this number is a W$^\ast$-invariant in general, the related property of strong $1$-boundedness introduced by Jung does satisfy this.

\begin{definition}[\cite{MR2373014}]
Let $r \in \mathbb{R}$, then we call $X$ an \emph{$r$-bounded set} if and only if there exists a constant $K \geq 0$ such that for sufficiently small $\varepsilon$ we have the estimate
\begin{align*}
    \chi(X + \varepsilon S : S) \leq (r-n) \abs{\log(\varepsilon)} + K .
\end{align*}
We call $X$ a \emph{strongly $1$-bounded set} if and only if it is a $1$-bounded set and it contains an element with finite free entropy (relative to the empty set).
\end{definition}

Note that being $r$-bounded is a strengthening of the inequality $\delta_0(X) \leq r$.
For a self-adjoint element to have finite free entropy it is sufficient that its spectral measure with respect to $\tau$ admits a bounded density with respect to the Lebesgue measure.

\begin{definition}[\cite{MR2373014}]
A finite von Neumann algebra $\cM$  has \emph{property (J)} if any finite set of self-adjoint generators is $1$-bounded, and is \emph{strongly $1$-bounded} if it admits a strongly $1$-bounded generating set.
\end{definition}

Jung proved a remarkable result that any tracial von Neumann algebra admitting a strongly $1$-bounded generating set has property (J) \cite{MR2373014}.
In particular, these properties are equivalent for finitely generated von Neumann algebras.

Now, let us note the following key technical ingredient, which is a slight modification of \cite{MR2373014}*{Theorem 5.3}.

\begin{proposition}\label{JP-OF}
Let $\mathcal A\subset \mathcal B$ be an inclusion of tracial von Neumann algebras with $\mathcal A$ strongly $1$-bounded.
Let $\{u_j\}_{j=1}^\infty$ be a sequence of unitaries in $B$, and write $\mathcal A_0 = \mathcal A$ and $\mathcal A_k$ for the von Neumann algebra generated by $\mathcal A$ and $\{u_1,\dots,u_k\}$.
Assume that for every $j$ there is a diffuse self-adjoint element $y_j \in \mathcal A_{j-1}$ such that $u_j y_j u_j^* \in \mathcal A_{j-1}$.
Denote by $\mathcal A_\infty$ the von Neumann algebra generated by $\mathcal A$ and all of the $u_j$'s.
Then $\mathcal A_\infty$ has property (J).
\end{proposition}

\begin{proof}
The only difference from \cite{MR2373014}*{Theorem 5.3} is that, instead of assuming $u_j u_{j-1} u_j^* \in \mathcal A_{j - 1}$ and that $u_j$ is diffuse (the latter implicit in \cite{MR2373014}), we allow $u_j$ to have atoms and assume another diffuse self-adjoint element $y_j \in \mathcal A_{j-1}$ for the conjugation by $u_j$.
This is still enough to have essentially the same proof as in \cite{MR2373014}*{Theorem 5.3}, noting that \cite{MR2373014}*{Lemmas 5.1 and 5.2} apply.
\end{proof}

Another ingredient we need is the permanence of strong $1$-boundedness under taking amplifications, as follows.

\begin{proposition}[\cite{arXiv:2107.03278}, cf.\citelist{\cite{MR2373014}*{Corollary 3.6}\cite{MR3801429}*{Proposition A.13(ii)}}]\label{S1B-Amp}
Let $\cM$ be a strongly $1$-bounded II$_1$-factor.
For any $t > 0$, the amplification $\cM^t$ is again strongly $1$-bounded.
\end{proposition}

Next let us recall some concepts from the theory of subfactors, see \cite{MR1642584} and references therein for details.
Let $\cN \subset \cM$ be a finite index inclusion of II$_1$-factors.
Then, the orthogonal projection $e_1\colon L^2(\cM) \to L^2(\cN)$ is called the \emph{Jones projection} associated with this inclusion.
The von Neumann algebra $\cM_1$ generated by $\cM$ and $e_1$ is again a II$_1$-factor, and the inclusion $\cM \subset \cM_1$, called \emph{the basic construction}, has the same index as $\cN \subset \cM$.
Moreover, $\cM_1$ is the commutant of the right multiplication action of $\cN$ on $L^2(\cM)$.
Iterating this construction, we obtain \emph{the Jones tower}
\begin{align*}
  \cM_{-1} &\subset \cM_0 \subset \cM_1 \subset \dots, &
  \cM_{-1} &= \cN,&
  \cM_0 &= \cM.&
\end{align*}
We denote the Jones projection of $\cM_{i - 1} \subset \cM_i$ by $e_i \in B(L^2(\cM_{i-1}))$, so that $\cM_{i}$ is generated by $\cM_{i-1}$ and  $e_i$.

A key observation is that $\cM_{i-2}$ and $\cM_i^{\opo}$ are commutants of each other on $L^2(\cM_{i-1})$, so that $\cM_i$ and $\cM_{i-2}$ are amplifications of each other.
In particular, $\cM_{2i}$ is always an amplification of $\cM$, while $\cM_{2i-1}$ is always an amplification of $\cN$.

There is also a subfactor $\cN_1 \subset \cN$ such that $\cN \subset \cM$ is isomorphic to its basic construction.
Repeating this construction, we obtain \emph{the Jones tunnel}
\begin{align*}
  \cN_{-1} &\supset \cN_0 \supset \cN_1 \supset \dots ,&
  \cN_{-1} &= \cM,&
  \cN_0 &= \cN ,
\end{align*}
and we obtain corresponding Jones projections $e_{-i} \in \cN_{i-1}$ for $i \geq 1$.
Again, $\cN_{2i}$ is an amplification of $\cN$, while $\cN_{2i-1}$ is an amplification of $\cM$.

\subsection{Strong \texorpdfstring{$1$}{1}-boundedness of quantum automorphism group factors}

We are now ready for the proof of the main technical result of this section.

\begin{theorem}\label{S1B-SFs}
Assume that we have a unital finite index inclusion $\cN\subset \cM$ of II$_1$-factors.
Then $\cN$ is strongly $1$-bounded if and only if $\cM$ is strongly $1$-bounded.
\end{theorem}

\begin{proof}
Assume that $\cM$ is strongly $1$-bounded.
To show that $\cN$ is strongly $1$-bounded, it is enough to have the same for $\cM_1$ by Proposition \ref{S1B-Amp}.

Take the first Jones projection $e_1$ as above, and set
\begin{align*}
u_1 &= 1 - 2 e_1,&
u_n &= 1 \quad (n > 1).
\end{align*}
We want to use Proposition \ref{JP-OF} for $A = \cM$, $B = \cM_2$, and $\{u_n \} \subset B$, to conclude that $\cM_1 = \cM \vee \{ u_n \mid n = 1, 2, \dots \}$ is strongly $1$-bounded.
It is enough to find a diffuse self-adjoint element $x \in \cM$ that commutes with $e_1$, as we would have
\begin{align*}
  u_1 x u_1^* = x \in \cM.
\end{align*}
Since $e_1$ is in the commutant of $\cN$, any choice of diffuse $x = x^* \in \cN$ will do.

Conversely, assume that $\cN$ is strongly $1$-bounded.
We can then simply apply the same argument as before.
Indeed, $\cM$ is generated by $\cN$ and the Jones projection for $\cN_1 \subset \cN$.
\end{proof}

\begin{corollary}
Assume that we have a unital finite index inclusion $\cN\subset \cM$ of II$_1$-factors.
If at least one of $\cN$ and $\cM$ is strongly $1$-bounded, then all of the II$_1$-factors in the Jones tower and tunnel are strongly $1$-bounded.
\end{corollary}

\begin{example}\label{ex:twist-by-J}
An interesting source of finite index inclusions comes from the graded twists of compact quantum groups introduced in \cite{MR3580173}.
Using this technique, one can realise $L^\infty(O_{2m}^{+J})$ as an index $4$ subfactor of $L^\infty(O_{2m}^+) \otimes M_2$ (Examples 3.3 and 2.17 in \cite{MR3580173}).
Concretely, one obtains $\cO(O_{2m}^{+J})$ as a Hopf subalgebra of $\cO(O_{2m}^+) \rtimes_\alpha \mathbb{Z}_2$ via the embedding $u_{ij}^J \mapsto u_{ij} \otimes g$, where $g$ is the generator of $\mathbb{Z}_2$ and $\alpha(u) = - J_{2m} u J_{2m}$.
One then takes the crossed product by the dual action.

It was shown in \cite{MR3783410} that $L^\infty(O_n^+)$ is strongly $1$-bounded for all $n \geq 3$.
Building upon their techniques, the case of $L^\infty(O_{2m}^{+J})$ was settled for all $m\geq 2$ in \cite{MR4288353}.
With Theorem \ref{S1B-SFs} and the graded twist technique in hand, we obtain an alternative proof for the latter case.
\end{example}

\begin{corollary}\label{col:S1B-QAG}
Let $B$ be a C$^\ast$-algebra such that $\dim B = n^2$ with $n \geq 3$.
Then $L^\infty(\Aut^+(B))$ is a strongly $1$-bounded II$_1$-factor.
\end{corollary}

\begin{proof}
The factoriality of $L^\infty(\Aut^+(B))$ is proved in \cite{MR3138849}, so it remains to prove the strong $1$-boundedness.

Since $L^\infty( \Aut^+( M_n ) )$ is an index $2$ subfactor of $L^\infty(O_n^+)$ \cite{MR3138849}, it must be strongly $1$-bounded as soon as $n \geq 3$. By Corollary \ref{cor:dbl-cross-prod-is-mat-amp}, $L^\infty( S_{n^2}^+ )$ is a finite index subfactor of $L^\infty( \Aut^+( M_n ) ) \otimes M_k$ for some $k$.
Hence $L^\infty(S_{n^2}^+)$ must be strongly $1$-bounded.
As for the general $B$ with $\dim B = n^2$, again by Corollary  \ref{cor:dbl-cross-prod-is-mat-amp} we have a finite index inclusion of $L^\infty( S_{n^2}^+ )$ in $L^\infty( \Aut^+( B ) ) \otimes M_\ell$ for some $\ell$, hence $L^\infty( \Aut^+( B ) )$ must be strongly $1$-bounded.
\end{proof}

Similar arguments applied to embeddings for the linking algebras discussed in Remark \ref{rem:FIELA} give the following additional corollary.

\begin{corollary}
Let $B_i$ for $i=1,2$ be finite dimensional C$^\ast$-algebras such that $\dim B_1 = \dim B_2 = n^2$ with $n \geq 3$.
Then the linking algebra $\cO_{12} = \cO(\Aut^+(B_1),\Aut^+(B_2))$ is non-trivial.
Let $\mathcal{L}_{12}$ be the tracial von Neumann algebra coming from the GNS-construction for the trace $\omega_{12}$ on $\cO_{12}$.
Then $\mathcal{L}_{12}$ is strongly $1$-bounded.
\end{corollary}

\subsection{Lack of strong \texorpdfstring{$1$}{1}-boundedness for free unitary group factors}

We next show that the free unitary quantum groups $U_n^+$ give II$_1$-factors $L^\infty(U_n^+)$ that are not strongly $1$-bounded, in contrast to the quantum groups $O_m^+$ and $\Aut^+(B)$.

Let us begin with a few remarks on the behaviour of $r$-boundedness under algebraic manipulations and free products.
It is well known that the value of the free entropy dimension of some finite tuple of self-adjoint elements in a tracial von Neumann algebra depends only on the generated $\ast$-algebra \cite{MR1601878}.
We show that this also holds for $r$-boundedness.

\begin{proposition}
Let $(A,\tau)$ be a tracial von Neumann algebra and let $X = (X_1,\dots,X_n)$ and $Y = (Y_1,\dots,Y_m)$ be finite tuples of self-adjoint elements of $A$.
Assume that $X$ and $Y$ generate the same $\ast$-algebra.
Then $X$ is $r$-bounded if and only if $Y$ is $r$-bounded.
\end{proposition}

\begin{proof}
Assume that $X$ is $r$-bounded.
By assumption there exist noncommutative polynomials $P_1,\dots,P_m$ in $n$ variables such that $Y_j = P_j(X)$ for $1\leq j \leq m$.
It is clear that the conditions of \cite{MR4288353}*{Lemma 4.1} are met, and so we can conclude that $X \cup Y$ is $r$-bounded.
But then \cite{MR2373014}*{Lemma 3.1} implies that $Y$ is $r$-bounded.
By symmetry, we are done.
\end{proof}

We now investigate free products of $r_i$-bounded sets.
We get at least one estimate for free from the subadditivity of relative free entropy.

\begin{lemma}
Let $X^{(1)},\dots,X^{(i)}$ for some $i \geq 2$ be such that $X^{(j)}$ is an $r_j$-bounded set for $1 \leq j \leq i$.
Then $\cup_j X^{(j)}$ is $(r_1 + \dots + r_i)$-bounded.
\end{lemma}

Of course, if the sets $X^{(j)}$ are free, one expects this to be the optimal level of boundedness (provided each $r_j$ was optimal).
Here optimal is meant in the sense that
\begin{align*}
    \delta_0 \left( \bigcup_j X^{(j)} \right) = r_1 + \dots + r_i
\end{align*}
This is not known in general even for just the free entropy dimension.
However, the following does hold by \cite{MR2439665}.

\begin{proposition}
Let $\mathcal A_i$ for $1 \leq i \leq n$ be Connes embeddable diffuse finite von Neumann algebras.
Assume that $X^{(i)}$ generates $\mathcal A_i$ and $\delta_0(X^{(i)}) = 1$ for all $i$.
Then in $\ast_{i=1}^n \mathcal A_i$ we have $\delta_0(\bigcup X^{(i)}) = n$.
\end{proposition}

\begin{proof}
This follows from \cite{MR2439665}*{Proposition 2.4 and Corollary 4.8}.
\end{proof}

\begin{proposition} \label{P:nots1b}
For any $n \geq 2$ and $m \geq 3$, the II$_1$-factors $L^\infty(U_n^+)$ and $L^\infty(O_m^+)$ are not isomorphic.
\end{proposition}

\begin{proof}
Since $L^\infty( U_2^+ ) \cong \mathcal{L}\mathbb{F}_2$, which is not strongly $1$-bounded, we may assume that $n \geq 3$.

Consider $L^\infty(O_n^+ \ast O_n^+)$, and let $\mathcal{U}^{(k)} = \{u_{ij}^{(k)}\}_{i,j=1}^n$ with $k=1,2$ be the two free sets of matrix coefficients of the fundamental representations.
Then $\delta_0(\mathcal{U}^{(1)} \cup \mathcal{U}^{(2)}) = 2$.
We now use again the graded twist technique of \cite{MR3580173} (compare with Example \ref{ex:twist-by-J} above).
In their Examples 3.6 and 2.18 it is shown how to realise $U_n^+$ as a graded twist of $O_n^+ \ast O_n^+$ by $\mathbb{Z}_2$.
The $\mathbb{Z}_2$ action on $O_n^+ \ast O_n^+$ is the one that swaps the two free copies of the generators of $O_n^+$.
Taking once again the crossed product by the dual action, it follows that $L^\infty(U_n^+)$ appears as a finite index subfactor of $L^\infty(O_n^+\ast O_n^+) \otimes M_2$, and hence cannot be strongly $1$-bounded.
\end{proof}

\section{Unitary error bases and crossed product isomorphisms}
\label{sec:unitary-error-basis}

In this final section we outline how the existence of ``small'' finite-dimensional representations of the linking algebra $\cO(\Aut^+(B), S_N^+)$ as in Proposition \ref{prop:fin-dim-rep-link-Autplus-SNplus} is intimately related to the construction of unitary error bases in quantum information theory.
Note that a finite dimensional representation of the linking algebra can be interpreted as the \emph{isomorphism game} between $B$ and $\bC^N$ having a quantum winning strategy.
In fact, it was the study of quantum graph colorings in \cite{BGH20} and their implementation via generalized unitary error bases that initially inspired the authors to the concrete crossed product isomorphisms for quantum automorphism groups that are derived in this section.

\begin{definition} \label{def:UEB}
Let $n \in \bN$.
A \emph{unitary error basis} is a basis $\{u_a\}_{a=1}^{n^2}$ of $M_n$ consisting of unitary matrices that are orthogonal with respect to the normalized trace inner product:
\[
\tr(u_a^*u_b) = \delta_{a,b} \quad (1 \le a,b \le n^2).
\]
\end{definition}

We note that an equivalent characterization of a unitary error basis in $M_n$ is a family of unitaries $\{u_a\}_{a = 1}^{n^2}$ with the following \emph{depolarization property}:
\[
\sum_{a=1}^{n^2} u_a^*x u_a = n\Tr(x)1.
\]

Let $(\ket{i})_{i = 0}^n$ be a standard basis of $\bC^n$, and put
\begin{equation*}
\ket{\phi}=\frac{1}{\sqrt{n}}\sum_{i=0}^{n-1} \ket{i i} = \frac{1}{\sqrt{n}}\sum_{i=0}^{n-1} \ket{i} \otimes \ket{i} \in \bC^n \otimes \bC^n.
\end{equation*}
Let us further fix a primitive $n$-th root of unity $\omega$.
Then the \emph{generalized Pauli matrices} $X_n,Z_n \in M_n$ are defined to be
\begin{align}\label{eq:gen-pauli-matrices}
X_n\ket{j} &= \omega^j \ket{j},&
Z_n\ket{j} &= \ket{j+1},
\end{align}
where the index is computed modulo $n$.
We then put $T_{i,j}=X_n^i Z_n^j$ for $0 \leq i,j \leq n-1$, and
\[
\ket{\phi_{i,j}}=(T_{i,j} \otimes I_n)\ket{\phi}.
\]
The Pauli matrices satisfy the commutation relation $X_n Z_n=\omega Z_n X_n$, and we have $\tr(T_{i,j})= \delta_{i, 0} \delta_{j, 0}$.
Thus, $\{T_{i,j} \mid 0 \leq i,j \leq n-1\}$ is a unitary basis for $(M_n,\tr)$, called the \emph{Weyl unitary error basis}.
On the other hand, $\{\ket{\phi_{i,j}} \mid 0 \leq i,j \leq n-1\}$ is an orthonormal basis for $\bC^n \otimes \bC^n$, called a \emph{maximally entangled basis}.

\subsection{Finite dimensional representations of the linking algebra}

Let us begin with a concrete presentation of $\cO(\Aut^+(B))$.
Throughout the section we work with a multimatrix decomposition 
\[
B = \bigoplus_{r=1}^m M_{n_r},
\]
and denote the canonical matrix units for $M_{n_s}$ inside of $B$ by $E_{ij}^{(s)}$, $0 \leq i,j \leq n_s-1$.
Then $\cO(\Aut^+(B))$ is the universal unital $\ast$-algebra generated by elements $q_{(i,j),(k,\ell)}^{(s,r)}$, $1 \leq s,r \leq m$, $0 \leq i,j \leq n_s-1$, $0 \leq k,\ell \leq n_r-1$, satisfying
\begin{enumerate}
\item
$\sum_{v=0}^{n_r-1} q_{(i,j),(k,v)}^{(s,r)}q_{(i',j'),(v,\ell)}^{(s',r)}=\delta_{ji'} \delta_{ss'} q_{(i,j'),(k,\ell)}^{(s,r)}$;
\item
$\sum_{v=0}^{n_s-1} n_s^{-1}q_{(i,v),(k,\ell)}^{(s,r)}q_{(v,j),(k',\ell')}^{(s,r')}=\delta_{\ell k'} \delta_{rr'} n_r^{-1}q_{(i,j),(k,\ell')}^{(s,r)}$;
\item
$q_{(i,j),(k,\ell)}^{(s,r)*}=q_{(j,i),(\ell,k)}^{(s,r)}$;
\item
$\sum_{s=1}^m \sum_{i=0}^{n_s-1} q_{(i,i),(k,\ell)}^{(s,r)}=\delta_{k\ell}$;
\item
$\sum_{r=1}^m \sum_{k=0}^{n_r-1} n_r q_{(i,j),(k,k)}^{(s,r)}=n_s\delta_{ij}$.
\end{enumerate}
The coproduct on $\cO(\Aut^+(B))$ is given by 
\begin{equation}
\label{eq:coprod-on-AutplusB}
\Delta(q_{(i,j),(i'j')}^{(r,s)})  = \sum_{z=1}^m \sum_{k,\ell = 0}^{n_z-1} q_{(i,j),(k,\ell)}^{(r,z)} \otimes q_{(k,\ell), (i',j')}^{(z,s)}.
\end{equation}

Then the coaction of $\cO(\Aut^+(B))$ on $B$ is given by
\[
\rho(E_{ij}^{(s)})=\sum_{r=1}^m \sum_{k,\ell=0}^{n_r-1} q_{(i,j),(k,\ell)}^{(s,r)} \otimes E_{k\ell}^{(r)}.
\]

We now proceed to present a finite-dimensional representation of the linking algebra $\cO(\Aut^+(B), S_{\dim B}^+)$.  Put $N = \dim B$, $d = n_1n_2\ldots n_m$, and identify $M_d \cong M_{n_1} \otimes \ldots M_{n_m}$ in the usual way.
For $T \in M_{n_r}$, we put
\[
T^{(r)}=I_{n_1} \otimes I_{n_2} \otimes \cdots \otimes I_{n_{r-1}} \otimes T \otimes I_{n_{r+1}} \otimes \cdots \otimes I_{n_m} \in M_d.
\]

For each $1 \le r \le m$, let us write the Weyl error basis of $M_{n_r}$ as $U_{r,i,j} = T_{i,j}$.
We then take, for $1 \le s \le m$, $0 \le a,b < n_s$, the elements 
\[
P_{s,a,b} = \sum_{i,j = 0}^{n_r-1} E_{i,j} \otimes P^{(s)}_{(a,b),(i,j)} \in M_{n_s} \otimes M_d \subset B \otimes M_d, \quad P_{(a,b),(i,j)} = \frac1{n_s} U^{*}_{s,a,b} E_{i,j} U_{s,a,b}.
\]

Now we can present a concrete representation of the linking algebra as in Proposition \ref{prop:fin-dim-rep-link-Autplus-SNplus}, but with a smaller dimension than the one given in its proof.
Write the standard basis of minimal projections in the Abelian C$^\ast$-algebra $\bC^N$ as $\{e_{s,a,b}\}$ for $1 \le s \le m$ and $0 \le a,b < n_s$.

\begin{proposition} \label{P:UEB rep}
Under the above setting, the map
\[
\hat\rho\colon \bC^N \to  B \otimes M_d, \quad e_{s,a,b} \mapsto P_{s,a,b}
\]
is a unital $\ast$-homomorphism that is Plancherel trace covariant.
In particular, the linking algebra $\cO(\Aut^+(B), S_N^+)$ admits a non-zero $\ast$-homomorphism $\pi$ to $M_d$ characterized by $\hat \rho = (\id  \otimes \pi)\rho$.
\end{proposition}

\begin{proof}
When $m = 1$, this result is exactly \cite{BEVW20}*{Proposition 7.2}.
The general case only requires some small modifications of the proof there.

Let us first fix $1 \le s \le m$.
First, the polarization property implies
\[
\sum_{a,b} P_{s,a,b}  = I_{n_s} \otimes I_d.
\]
Next, each $P_{s,a,b}$ is a projection by a standard calculation.
Moreover, the orthogonality of $(U_{s,i,j})_{i,j}$ implies that the $P_{s,a,b}$ and $P_{s,a',b'}$ are mutually orthogonal.
Since we also have $P_{s,a,b}P_{s',a',b'} = 0$ for all $s \ne s'$, it follows that
\[
\hat \rho\colon \bC^N \to  B \otimes M_d, \quad \hat \rho(e_{s,a,b}) = P_{s,a,b}
\]
defines a $\ast$-homomorphism.  

We also have
\[
(n_s\Tr_{n_s} \otimes \id)(P_{s,a,b}) = I_d,
\]
hence $(\psi \otimes \id)(P_{s,a,b}) = \frac1N = \psi(e_{s,a,b})$.
This shows the compatibility with Plancherel traces.  The existence of $\pi$ then follows from the universal properties of the algebras under consideration. 
\end{proof}

Now, let $u_{(s,x,y),(r,v,w)}$ denote the matrix coefficients of $\cO(S_N^+)$ corresponding to the basis $(e_{s,x,y})_{s,x,y}$ of $\bC^N$. For each $1 \leq s \leq m$, we fix a primitive $n_s$-th root of unity $\omega_{n_s}$.

\begin{theorem}
\label{theorem: finite index embeddings}
In the above setting, there are unital $\ast$-homomorphisms
\begin{align*}
\pi\colon \cO(\Aut^+(B)) &\to M_d \otimes M_d \otimes \cO(S_N^+),&
\rho\colon \cO(S_N^+) &\to M_d \otimes M_d \otimes \cO(\Aut^+(B)),
\end{align*}
for $d=n_1 \cdots n_m$, that are compatible with traces induced by the Haar traces (that is $(\id \otimes \otimes \id \otimes h_{S_N^+})\pi = h_{\Aut^+(B)}(\cdot ) (1 \otimes 1)$ and $(\id \otimes \otimes \id \otimes h_{\Aut^+(B)})\rho = h_{S_N^+}(\cdot ) (1 \otimes 1)$).
These morphisms are characterized by
\begin{align*}
\pi(q_{(i,j),(k,\ell)}^{(s,r)})&=\frac{1}{n_s} \sum_{x,y=0}^{n_s-1} \sum_{v,w=0}^{n_r-1} \omega_{n_s}^{-x(i-j)}\omega_{n_r}^{-v(k-\ell)} E_{i-y,j-y}^{(s)} \otimes E_{k-w,\ell-w}^{(r)} \otimes u_{(s,x,y),(r,v,w)}, \\
\rho(u_{(s,x,y),(r,v,w)})&=\frac{1}{n_r}\sum_{i,j=0}^{n_s-1} \sum_{k,\ell=0}^{n_r-1} \omega_{n_s}^{x(i-j)}\omega_{n_r}^{v(k-\ell)} E_{i-y,j-y}^{(s)} \otimes E_{k-w,\ell-w}^{(r)} \otimes q_{(i,j),(k,\ell)}^{(s,r)}.
\end{align*}
\end{theorem}

\begin{proof}
As remarked in the proof of Theorem \ref{T: transferring matrix models CEP}, the existence of a non-zero representation $\cO(\Aut^+(B), S_N^+) \to M_d$ canonically gives rise to trace-preserving embeddings $\pi$, $\rho$ with the correct domains and ranges.  The specific form of $\pi,\rho$ described in the statement of the present theorem follows if we choose to use the Weyl unitary error bases in the representation $\cO(\Aut^+(B), S_N^+) \to M_d$ supplied by Proposition \ref{P:UEB rep}. 
\end{proof}

Again, it is convenient to write the homomorphisms $\pi$ and $\rho$ in a different form.

Let $X_{n_r}, Z_{n_r} \in M_{n_r}$ be generalized Pauli matrices as in \eqref{eq:gen-pauli-matrices}.
In the tensor product
\[
M_{n_1} \otimes M_{n_1} \otimes M_{n_2} \otimes M_{n_2} \otimes \cdots \otimes M_{n_m} \otimes M_{n_m},
\]
we write $\varphi_{ij}^{[s]}$ for the projection
\[
I_{n_1} \otimes I_{n_1} \otimes \cdots \otimes I_{n_{s-1}} \otimes I_{n_{s-1}} \otimes \ket{\phi_{i,j}}\bra{\phi_{i,j}} \otimes I_{n_{s+1}} \otimes I_{n_{s+1}} \otimes \cdots \otimes I_{n_m} \otimes I_{n_m}.
\]
Then the projections $\{ \varphi_{ij}^{[s]} \mid 0 \leq i,j \leq n_s-1, \, 1 \leq s \leq m\}$ form a projection valued measure with $N$ outcomes in this rearranged tensor product $M_d \otimes M_d$. 
Similarly, we write $T_{i,j}^{[s]}$ for the operator
\[
I_{n_1} \otimes I_{n_1} \otimes \cdots \otimes I_{n_{s-1}} \otimes I_{n_{s-1}} \otimes (X_{n_s}^iZ_{n_s}^j \otimes I_{n_s}) \otimes I_{n_{s+1}} \otimes I_{n_{s+1}} \otimes \cdots \otimes I_{n_m} \otimes I_{n_m}.
\]
We set $X^{[s]}=T_{1,0}^{[s]}$ and $Z^{[s]}=T_{0,1}^{[s]}$.

For each $1 \leq r,s \leq m$, we set $\displaystyle Q^{(s,r)}=\sum_{i,j=0}^{n_s-1} \sum_{k,\ell=0}^{n_r-1} E_{ij}^{(s)} \otimes E_{k\ell}^{(r)} \otimes q_{(i,j),(k,\ell)}^{(s,r)}$.
Then we compute
\begin{multline*}
(\id_{d} \otimes \id_{d}\otimes \pi)(Q^{(s,r)})\\
=\frac{1}{n_s}\sum_{i,j,k,\ell} \sum_{x,y,v,w} \omega_{n_s}^{-x(i-j)}E_{ij}^{(s)} \otimes \omega_{n_r}^{-v(k-\ell)}E_{k\ell}^{(r)} \otimes E_{i-y,j-y}^{(s)} \otimes E_{k-w,\ell-w}^{(r)} \otimes u_{(s,x,y),(r,v,w)}.
\end{multline*}
After a shuffle, we see that the image of $(\id_{d} \otimes \id_{d} \otimes \pi)(Q^{(s,r)})$ in $M_d^{\otimes 4} \otimes \cO(S_{\dim(B)}^+)$ is given by
\[
n_r\sum_{x,y=0}^{n_s-1} \sum_{v,w=0}^{n_r-1} \varphi_{-x,y}^{[s]} \otimes \varphi_{-v,w}^{[r]} \otimes u_{(s,x,y),(r,v,w)}.
\] 

Similarly, we have
\begin{align*}
\rho(u_{(s,x,y),(r,v,w)})&=\frac{1}{n_r} \sum_{i,j=0}^{n_r-1} \sum_{k,\ell=0}^{n_s-1} \omega_{n_s}^{x(i-j)} E_{i-y,j-y}^{(s)} \otimes  \omega_{n_r}^{v(k-\ell)} E_{k-w,\ell-w}^{(r)} \otimes q_{(i,j),(k,\ell)}^{(s,r)} \\
&=(T_{x,-y}^{[s]} \otimes T_{v,-w}^{[r]} \otimes 1)\left( \frac{1}{n_r} Q^{(s,r)} \right)(T_{x,-y}^{[s]} \otimes T_{v,-w}^{[r]} \otimes 1)^*.
\end{align*}

\subsection{Iterated Crossed Product Isomorphisms}

Next, let us establish a parallel of Corollary \ref{cor:dbl-cross-prod-is-mat-amp}, but with smaller group actions (using $\Gamma$ instead of $\Gamma^2$).
For $1 \leq t \leq m$, we define algebra automorphisms $\alpha_{1,t},\alpha_{2,t},\alpha_{3,t},\alpha_{4,t}$ on $\cO(\Aut^+(B))$ by
\begin{align*}
\alpha_{1,t}(q_{(i,j),(k,\ell)}^{(s,r)})&=
\begin{cases}
q_{(i,j),(k,\ell)}^{(s,r)} & (s \neq t) \\
\omega_{n_t}^{i-j} q_{(i,j),(k,\ell)}^{(s,r)} & (s=t), \end{cases}&
\alpha_{3,t}(q_{(i,j),(k,\ell)}^{(s,r)})&=
\begin{cases} q_{(i,j),(k,\ell)}^{(s,r)} & (r \neq t) \\
\omega_{n_r}^{k-\ell} q_{(i,j),(k,\ell)}^{(s,r)} & (r=t), \end{cases} \\
\alpha_{2,t}(q_{(i,j),(k,\ell)}^{(s,r)})&=
\begin{cases} q_{(i,j),(k,\ell)}^{(s,r)} & (s \neq t) \\
q_{(i+1,j+1),(k,\ell)}^{(s,r)}& (s=t), \end{cases}&
\alpha_{4,t}(q_{(i,j),(k,\ell)}^{(s,r)})&=
\begin{cases} q_{(i,j),(k,\ell)}^{(s,r)} & (r \neq t) \\
q_{(i,j),(k+1,\ell+1)}^{(s,r)} & (r=t). \end{cases}
\end{align*}
In $\alpha_2$ and $\alpha_4$, the shifts (if applicable) are done with indices modulo $n_t$.

From the above presentation, we see that $\alpha_{i, t}$ is an automorphisms of order $n_t$ for all $i$.
Moreover, $\alpha_{1,t}$ and $\alpha_{3,t'}$ commute for all $t,t'$ and $\alpha_{2,t}$ and $\alpha_{4,t'}$ commute for all $t,t'$. Thus, we have obtained group actions of the group $\Gamma = \prod_{t=1}^m (\bZ_{n_t} \times \bZ_{n_t})$ on $\cO(\Aut^+(B))$.

\begin{proposition}\label{prop:alpha-pres-haar-tr}
The actions $\alpha_{i, t}$ preserve the Haar trace of $\cO(\Aut^+(B))$.
\end{proposition}

\begin{proof}
By \eqref{eq:coprod-on-AutplusB}, the automorphism $\alpha_{1,t}$ is a right comodule endomorphism in the sense that
\[
(\alpha_{1,t} \otimes \id) \circ \Delta = \Delta \circ \alpha_{1,t}.
\]
Combined with the right invariance condition $(\id \otimes h) \circ \Delta(x) = h(x) 1$, we obtain $h \circ \alpha_{1, t} = h$.

The other cases are proved in the same way (with left equivariance and invariance for $\alpha_{3, t}$ and $\alpha_{4,t}$).
\end{proof}

Next, we construct the crossed product $\cO(\Aut^+(B)) \rtimes_{\alpha_1,\alpha_3} \Gamma$, which is the universal $\ast$-algebra generated by elements $q_{(i,j),(k,\ell)}^{(s,r)}$ and $z_{1,t},z_{3,t'}$, such that:
\begin{itemize}
\item
The elements $q_{(i,j),(k,\ell)}^{(s,r)}$ generate a copy of $\cO(\Aut^+(B))$;
\item
$z_{1,t}^d=z_{3,t}^d=1$ and $z_{1,t}^*=z_{1,t}^{d-1}$ and $z_{3,t}^*=z_{3,t}^{d-1}$; and $[z_{1,t},z_{3,t'}]=0$; and
\item
$z_{1,t}q_{(i,j),(k,\ell)}^{(s,r)}z_{1,t}^*=\alpha_{1,t}(q_{(i,j),(k,\ell)}^{(s,r)})$ and $z_{3,t}q_{(i,j),(k,\ell)}^{(s,r)}z_{3,t}^*=\alpha_{3,t}(q_{(i,j),(k,\ell)}^{(s,r)})$.
\end{itemize}
We consider the trace on this algebra induced by $h$.

The automorphisms $\alpha_{2,\tau},\alpha_{4,\tau'}$ on $\cO(\Aut^+(B))$ extend to trace-preserving automorphisms on $\cO(\Aut^+(B)) \rtimes_{\alpha_1,\alpha_3} \Gamma$ by setting
\begin{align*}
\alpha_{2,\tau}(z_{3,t})&=z_{3,t}, &
\alpha_{2,\tau}(z_{1,t})&=\begin{cases} z_{1,t} & \tau \neq t \\ \omega_{n_{\tau}}^{-1} z_{1,t} & \tau=t, \end{cases} \\
\alpha_{4,\tau}(z_{1,t})&=z_{1,t}, &
\alpha_{4,\tau}(z_{3,t})&=\begin{cases} z_{3,t} & \tau \neq t \\ \omega_{n_{\tau}}^{-1}z_{3,t} & \tau=t. \end{cases}
\end{align*}

\begin{theorem}
\label{T: crossed product isomorphism}
Under the above setting, the homomorphism $\pi$ in Theorem \ref{theorem: finite index embeddings} extends to a $\ast$-isomorphism 
\[
(\cO(\Aut^+(B)) \rtimes_{\alpha_1,\alpha_3} \Gamma) \rtimes_{\alpha_2,\alpha_4} \Gamma \to M_d \otimes M_d \otimes \cO(S_N^+)
\]
that intertwines the traces induced by the Haar traces.
\end{theorem}

\begin{proof}
We extend $\pi$ by the following formula on the generators of two copies of $\Gamma$:
\begin{align*}
\pi(z_{1,t})&=X_{n_t}^{(t)} \otimes I_d \otimes 1, &
\pi(z_{3,t})&=I_d \otimes X_{n_t}^{(t)} \otimes 1, \\
\pi(z_{2,t})&=Z_{n_t}^{(t)} \otimes I_d \otimes 1, &
\pi(z_{4,t})&=I_d \otimes Z_{n_t}^{(t)} \otimes 1,
\end{align*}
with generalized Pauli matrices as in \eqref{eq:gen-pauli-matrices}.
The compatibility with the original $\pi$ on $\cO(\Aut^+(B))$ follows from explicit presentation of $(\id_{d} \otimes \id_{d}\otimes \pi)(Q^{(s,r)})$ in the previous section.
\end{proof}

We thus see that, the Weyl unitary error basis of $M_d \otimes M_d$, together with the image of $\pi$ that was again defined through Weyl unitary bases of direct summands $M_{n_r}$ of $B$ and maximally entangled bases of $\bC^{n_r} \otimes \bC^{n_r}$, generate a matrix amplification of $\cO(S_N^+)$.

There is a similar extension of the homomorphism $\rho$.
On the algebra $\cO(S_N^+)$, we define $\ast$-algebra automorphisms $\beta_{\gamma,t}$, $1 \leq \gamma \leq 4$, $1 \leq t \leq m$, by
\begin{align*}
\beta_{1,t}(u_{(s,x,y),(r,v,w)})&=\begin{cases}
u_{(s,x,y),(r,v,w)} & (s \neq t) \\
u_{(s,x+1,y),(r,v,w)} & (s=t), \end{cases}\\
\beta_{3,t}(u_{(s,x,y),(r,v,w)})&=\begin{cases}
u_{(s,x,y),(r,v,w)} & (r \neq t) \\
u_{(s,x,y+1),(r,v,w)} & (r=t), \end{cases} \\
\beta_{2,t}(u_{(s,x,y),(r,v,w)})&=\begin{cases}
u_{(s,x,y),(r,v,w)} & (s \neq t \\
u_{(s,x,y),(r,v-1,w)}& (s=t), \end{cases}\\
\beta_{4,t}(u_{(s,x,y),(r,v,w)})&=\begin{cases}
u_{(s,x,y),(r,v,w)} & (r \neq t) \\
u_{(s,x,y),(r,v,w-1)} & (r=t). \end{cases}
\end{align*}
Evidently, each $\beta_{\gamma,t}$ is an automorphism on $\cO(S_N^+)$ of order $n_t$, while $\beta_{1,t} \circ \beta_{3,\tau}=\beta_{3,\tau} \circ \beta_{1,t}$ and $\beta_{2,t} \circ \beta_{4,\tau}=\beta_{4,\tau} \circ \beta_{2,t}$ for all $1 \leq t,\tau \leq m$.
We thus get a group action of $\Gamma$ on $\cO(S_N^+)$.
Again these preserve the Haar trace as in Proposition \ref{prop:alpha-pres-haar-tr}.

As before we form the crossed product
\[
\cO(S_N^+) \rtimes_{\beta_1,\beta_3} \Gamma,
\]
which is the universal $\ast$-algebra generated by elements $u_{(s,x,y),(r,v,w)}$, $1 \leq r,s \leq m$, $0 \leq x,y \leq n_s-1$, $0 \leq v,w \leq n_r-1$, and $z_{1,t}$, $z_{3,t}$, $1 \leq t \leq m$, satisfying the following:
\begin{itemize}
\item
the elements $u_{(s,x,y),(r,v,w)}$ satisfy the relations of the fundamental unitary in $\cO(S_N^+)$;
\item
$z_{1,t}^{n_t}=1=z_{3,t}^{n_t}$ and $z_{1,t}^*=z_{1,t}^{n_t-1}$ and $z_{3,t}^*=z_{3,t}^{n_t-1}$ for all $1 \leq t \leq m$; 
\item
$[z_{1,t},z_{3,\tau}]=0$ for all $t,\tau$; 
\item
$[z_{1,t},z_{1,\tau}]=[z_{3,t},z_{3,\tau}]=0$ for all $t,\tau$; and
\item
$z_{1,t}Az_{1,t}^*=\beta_{1,t}(A)$ and $z_{3,t}Az_{3,t}^*=\beta_{3,t}(A)$ for all $A \in \cO(S_N^+)$.
\end{itemize}

The automorphisms $\beta_{2,t},\beta_{4,t}$ extend to $\cO(S_N^+) \rtimes_{\beta_1,\beta_3} \Gamma$ by setting
\begin{align*}
\beta_{2,t}(z_{3,\tau})&=z_{3,\tau} \quad \forall 1 \leq \tau \leq m,&
\beta_{2,t}(z_{1,\tau})&=\begin{cases} z_{1,\tau} & t \neq \tau \\ \omega_{n_t}^{-1} z_{1,\tau} & t=\tau \end{cases} \\
\beta_{4,t}(z_{1,\tau})&=z_{1,\tau}, \quad \forall 1 \leq \tau \leq m,&
\beta_{4,t}(z_{3,\tau})&=\begin{cases} z_{3,\tau} & t \neq \tau \\ \omega_{n_t}^{-1} z_{3,\tau} & t=\tau. \end{cases}
\end{align*}

\begin{theorem}\label{T: crossed product isomorphism 2}
Under the above setting, the homomorphism $\rho$ in Theorem \ref{theorem: finite index embeddings} extends to a $\ast$-isomorphism 
\[
\cO(S_N^+) \rtimes_{\beta_1,\beta_3} \Gamma \rtimes_{\beta_2,\beta_4} \Gamma \to M_d \otimes M_d \otimes \cO(\Aut^+(B))
\]
that intertwines the traces induced by the Haar traces.
\end{theorem}

\begin{proof}
We get this by defining $\rho(z_{i,t})$ by the same formula as in the proof of Theorem \ref{T: crossed product isomorphism}.
\end{proof}

\raggedright


\begin{bibdiv}
\begin{biblist}

\bib{MR2681256}{article}{
      author={Andruskiewitsch, Nicol\'{a}s},
      author={Bichon, Julien},
       title={Examples of inner linear {H}opf algebras},
        date={2010},
        ISSN={0041-6932},
     journal={Rev. Un. Mat. Argentina},
      volume={51},
      number={1},
       pages={7\ndash 18},
         url={https://mathscinet.ams.org/mathscinet-getitem?mr=2681256},
      review={\MR{2681256}},
}

\bib{MR1378260}{article}{
   author={Banica, Teodor},
   title={Th\'{e}orie des repr\'{e}sentations du groupe quantique compact libre
   ${\rm O}(n)$},
   language={French, with English and French summaries},
   journal={C. R. Acad. Sci. Paris S\'{e}r. I Math.},
   volume={322},
   date={1996},
   number={3},
   pages={241--244},
   issn={0764-4442},
   review={\MR{1378260}},
}

\bib{Ba97}{article}{
      author={Banica, Teodor},
       title={Le groupe quantique compact libre {${\rm U}(n)$}},
        date={1997},
        ISSN={0010-3616},
     journal={Comm. Math. Phys.},
      volume={190},
      number={1},
       pages={143\ndash 172},
  eprint={\href{http://arxiv.org/abs/math/9901042v5}{\texttt{arXiv:math/9901042v5
  [math.QA]}}},
         url={http://dx.doi.org/10.1007/s002200050237},
         doi={10.1007/s002200050237},
      review={\MR{1484551 (99k:46095)}},
}


\bib{MR1709109}{article}{
      author={Banica, Teodor},
       title={Symmetries of a generic coaction},
        date={1999},
        ISSN={0025-5831},
     journal={Math. Ann.},
      number={4},
       pages={763\ndash 780},
         url={http://dx.doi.org/10.1007/s002080050315},
         doi={10.1007/s002080050315},
}

\bib{MR2492990}{article}{
      author={Banica, Teodor},
      author={Bichon, Julien},
       title={Quantum groups acting on 4 points},
        date={2009},
        ISSN={0075-4102},
     journal={J. Reine Angew. Math.},
      volume={626},
       pages={75\ndash 114},
         url={http://dx.doi.org/10.1515/CRELLE.2009.003},
         doi={10.1515/CRELLE.2009.003},
      review={\MR{2492990 (2010c:46153)}},
}

\bib{bb-inner}{article}{
      author={Banica, Teodor},
      author={Bichon, Julien},
       title={Hopf images and inner faithful representations},
        date={2010},
        ISSN={0017-0895},
     journal={Glasg. Math. J.},
      volume={52},
      number={3},
       pages={677\ndash 703},
  url={https://doi-org.offcampus.lib.washington.edu/10.1017/S0017089510000510},
         doi={10.1017/S0017089510000510},
      review={\MR{2679923}},
}

\bib{BaBiCu11}{article}{
      author={Banica, Teodor},
      author={Bichon, Julien},
      author={Curran, Stephen},
       title={Quantum automorphisms of twisted group algebras and free
  hypergeometric laws},
        date={2011},
        ISSN={0002-9939},
     journal={Proc. Amer. Math. Soc.},
      volume={139},
      number={11},
       pages={3961\ndash 3971},
         url={https://doi.org/10.1090/S0002-9939-2011-10877-3},
         doi={10.1090/S0002-9939-2011-10877-3},
      review={\MR{2823042}},
}


\bib{BhBrChWa17}{article}{
      author={Bhattacharya, Angshuman},
      author={Brannan, Michael},
      author={Chirvasitu, Alexandru},
      author={Wang, Shuzhou},
       title={Property ({T}), property ({F}) and residual finiteness for
  discrete quantum groups},
        date={2020},
        ISSN={1661-6952},
     journal={J. Noncommut. Geom.},
      volume={14},
      number={2},
       pages={567\ndash 589},
         url={https://doi.org/10.4171/jncg/373},
         doi={10.4171/jncg/373},
      review={\MR{4130839}},
}

\bib{MR1975007}{article}{
      author={Biane, P.},
      author={Capitaine, M.},
      author={Guionnet, A.},
       title={Large deviation bounds for matrix {B}rownian motion},
        date={2003},
        ISSN={0020-9910},
     journal={Invent. Math.},
      volume={152},
      number={2},
       pages={433\ndash 459},
         url={http://dx.doi.org/10.1007/s00222-002-0281-4},
         doi={10.1007/s00222-002-0281-4},
      review={\MR{1975007}},
}

\bib{bichon-lect-note-gal-ext-cpt-qgrp}{misc}{
      author={Bichon, Julien},
       title={Galois extension for a compact quantum group},
         how={manuscript},
        note={lecture note available at the author's website}
}

\bib{MR3049699}{article}{
      author={Bichon, Julien},
       title={Hochschild homology of {H}opf algebras and free
  {Y}etter-{D}rinfeld resolutions of the counit},
        date={2013},
        ISSN={0010-437X},
     journal={Compos. Math.},
      volume={149},
      number={4},
       pages={658\ndash 678},
         url={https://doi.org/10.1112/S0010437X12000656},
         doi={10.1112/S0010437X12000656},
      review={\MR{3049699}},
}

\bib{MR2202309}{article}{
      author={Bichon, Julien},
      author={De~Rijdt, An},
      author={Vaes, Stefaan},
       title={Ergodic coactions with large multiplicity and monoidal
  equivalence of quantum groups},
        date={2006},
        ISSN={0010-3616},
     journal={Comm. Math. Phys.},
      volume={262},
      number={3},
       pages={703\ndash 728},
  eprint={\href{http://arxiv.org/abs/math/0502018}{\texttt{arXiv:math/0502018
  [math.OA]}}},
         url={http://dx.doi.org/10.1007/s00220-005-1442-2},
         doi={10.1007/s00220-005-1442-2},
      review={\MR{2202309 (2007a:46072)}},
}

\bib{MR3784753}{article}{
  author     = {Bichon, Julien},
  author     = {Kyed, David},
  author     = {Raum, Sven},
  journal    = {Canad. Math. Bull.},
  title      = {Higher {$\ell^2$}-{B}etti numbers of universal quantum groups},
  year       = {2018},
  issn       = {0008-4395},
  number     = {2},
  pages      = {225--235},
  volume     = {61},
  doi        = {10.4153/CMB-2017-036-3},
  url        = {https://doi-org.ezproxy.uio.no/10.4153/CMB-2017-036-3},
}

\bib{MR3580173}{article}{
      author={Bichon, Julien},
      author={Neshveyev, Sergey},
      author={Yamashita, Makoto},
       title={Graded twisting of categories and quantum groups by group
  actions},
        date={2016},
        ISSN={0373-0956},
     journal={Ann. Inst. Fourier (Grenoble)},
      volume={66},
      number={6},
       pages={2299\ndash 2338},
         url={http://dx.doi.org/10.5802/aif.3064},
         doi={10.5802/aif.3064},
}

\bib{MR3194750}{article}{
      author={Bichon, Julien},
      author={Yuncken, Robert},
       title={Quantum subgroups of the compact quantum group {$\rm
  SU_{-1}(3)$}},
        date={2014},
        ISSN={0024-6093},
     journal={Bull. Lond. Math. Soc.},
      volume={46},
      number={2},
       pages={315\ndash 328},
      eprint={\href{http://arxiv.org/abs/1306.6244}{\texttt{arXiv:1306.6244
  [math.QA]}}},
         url={http://dx.doi.org/10.1112/blms/bdt105},
         doi={10.1112/blms/bdt105},
      review={\MR{3194750}},
}
%
%\bib{MR3218183}{book}{
%      author={Brannan, Michael},
%       title={On the reduced operator algebras of free quantum groups},
%   publisher={ProQuest LLC, Ann Arbor, MI},
%        note={Thesis (Ph.D.)--Queen's University (Canada)},
%        year={2012},
%        ISBN={978-0499-28142-5},
%         url={http://gateway.proquest.com.ezproxy.uio.no/openurl?url_ver=Z39.88-2004&rft_val_fmt=info:ofi/fmt:kev:mtx:dissertation&res_dat=xri:pqm&rft_dat=xri:pqdiss:NS28142},
%}

\bib{MR2995437}{article}{
  author     = {Brannan, Michael},
  journal    = {J. Reine Angew. Math.},
  title      = {Approximation properties for free orthogonal and free unitary quantum groups},
  year       = {2012},
  issn       = {0075-4102},
  pages      = {223--251},
  volume     = {672},
  doi        = {10.1515/crelle.2011.166},
  url        = {https://doi-org.ezproxy.uio.no/10.1515/crelle.2011.166},
}

\bib{MR3138849}{article}{
      author={Brannan, Michael},
       title={Reduced operator algebras of trace-preserving quantum
  automorphism groups},
        date={2013},
        ISSN={1431-0635},
     journal={Doc. Math.},
      volume={18},
       pages={1349\ndash 1402},
         url={https://mathscinet.ams.org/mathscinet-getitem?mr=3138849},
      review={\MR{3138849}},
}

\bib{BrChFr18}{article}{
      author={Brannan, Michael},
      author={Chirvasitu, Alexandru},
      author={Freslon, Amaury},
       title={Topological generation and matrix models for quantum reflection
  groups},
        date={2020},
        ISSN={0001-8708},
     journal={Adv. Math.},
      volume={363},
       pages={106982, 31},
         url={https://doi.org/10.1016/j.aim.2020.106982},
         doi={10.1016/j.aim.2020.106982},
      review={\MR{4053516}},
}

\bib{BCV}{article}{
      author={Brannan, Michael},
      author={Collins, Beno\^\i{t}},
      author={Vergnioux, Roland},
       title={The {C}onnes embedding property for quantum group von {N}eumann
  algebras},
        date={2017},
        ISSN={0002-9947},
     journal={Trans. Amer. Math. Soc.},
      volume={369},
      number={6},
       pages={3799\ndash 3819},
         url={https://doi-org.offcampus.lib.washington.edu/10.1090/tran/6752},
         doi={10.1090/tran/6752},
      review={\MR{3624393}},
}

\bib{BEVW20}{misc}{
      author={Brannan, Mike},
      author={Eifler, Kari},
      author={Voigt, Christian},
      author={Weber, Moritz},
       title={Quantum {C}untz--{K}rieger algebras},
         how={preprint},
        date={2020},
      eprint={\href{http://arxiv.org/abs/2009.09466}{\texttt{arXiv:2009.09466
  [math.OA]}}},
        note={to appear in Trans. Amer. Math. Soc.},
}


\bib{BGH20}{misc}{
      author={Brannan, Michael},
      author={Ganesan, Priyanga},
      author={Harris, Samuel~J.},
       title={The quantum-to-classical graph homomorphism game},
         how={preprint},
        date={2020},
      eprint={\href{http://arxiv.org/abs/2009.07229}{\texttt{arXiv:2009.07229
  [math.OA]}}},
}


%\bib{BHTT21}{misc}{
%      author={Brannan, Michael},
%      author={Harris, Samuel~J.},
%      author={Todorov, Ivan~G.},
%      author={Turowska, Lyudmila},
%       title={Synchronicity for quantum non-local games},
%         how={preprint},
%        date={2021},
%      eprint={\href{http://arxiv.org/abs/2106.11489}{\texttt{arXiv:2106.11489
%  [math.OA]}}},
%}

\bib{MR3783410}{article}{
      author={Brannan, Michael},
      author={Vergnioux, Roland},
       title={Orthogonal free quantum group factors are strongly 1-bounded},
        date={2018},
        ISSN={0001-8708},
     journal={Adv. Math.},
      volume={329},
       pages={133\ndash 156},
         url={http://dx.doi.org/10.1016/j.aim.2018.02.007},
         doi={10.1016/j.aim.2018.02.007},
}

\bib{MR2439665}{article}{
      author={Brown, Nathanial},
      author={Dykema, Kenneth},
      author={Jung, Kenley},
       title={Free entropy dimension in amalgamated free products},
        date={2008},
        ISSN={0024-6115},
     journal={Proc. Lond. Math. Soc. (3)},
      volume={97},
      number={2},
       pages={339\ndash 367},
         url={http://dx.doi.org/10.1112/plms/pdm054},
         doi={10.1112/plms/pdm054},
        note={With an appendix by Wolfgang L\"{u}ck},
}

\bib{MR4211088}{article}{
  author   = {Caspers, Martijn},
  journal  = {Math. Ann.},
  title    = {Gradient forms and strong solidity of free quantum groups},
  year     = {2021},
  issn     = {0025-5831},
  number   = {1-2},
  pages    = {271--324},
  volume   = {379},
  doi      = {10.1007/s00208-020-02109-y},
  url      = {https://doi-org.ezproxy.uio.no/10.1007/s00208-020-02109-y},
}

\bib{MR3263140}{article}{
      author={Chirvasitu, Alexandru},
       title={Cosemisimple {H}opf algebras are faithfully flat over {H}opf
  subalgebras},
        date={2014},
        ISSN={1937-0652},
     journal={Algebra Number Theory},
      volume={8},
      number={5},
       pages={1179\ndash 1199},
         url={http://dx.doi.org/10.2140/ant.2014.8.1179},
         doi={10.2140/ant.2014.8.1179},
      review={\MR{3263140}},
}

\bib{chi-rfd}{article}{
      author={Chirvasitu, Alexandru},
       title={Residually finite quantum group algebras},
        date={2015},
        ISSN={0022-1236},
     journal={J. Funct. Anal.},
      volume={268},
      number={11},
       pages={3508\ndash 3533},
  url={https://doi-org.offcampus.lib.washington.edu/10.1016/j.jfa.2015.01.013},
         doi={10.1016/j.jfa.2015.01.013},
      review={\MR{3336732}},
}

\bib{MR2180603}{article}{
      author={Connes, Alain},
      author={Shlyakhtenko, Dimitri},
       title={{$L\sp 2$}-homology for von {N}eumann algebras},
        date={2005},
        ISSN={0075-4102},
     journal={J. Reine Angew. Math.},
      volume={586},
       pages={125\ndash 168},
      review={\MR{MR2180603 (2007b:46104)}},
}

\bib{MR2806547}{article}{
      author={De~Commer, Kenny},
       title={Galois objects and cocycle twisting for locally compact quantum
  groups},
        date={2011},
        ISSN={0379-4024},
     journal={J. Operator Theory},
      volume={66},
      number={1},
       pages={59\ndash 106},
      review={\MR{2806547}},
}

\bib{DeFrYa14}{article}{
      author={De~Commer, Kenny},
      author={Freslon, Amaury},
      author={Yamashita, Makoto},
       title={C{CAP} for universal discrete quantum groups},
        date={2014},
        ISSN={0010-3616},
     journal={Comm. Math. Phys.},
      volume={331},
      number={2},
       pages={677\ndash 701},
         url={https://doi.org/10.1007/s00220-014-2052-7},
         doi={10.1007/s00220-014-2052-7},
        note={With an appendix by Stefaan Vaes},
      review={\MR{3238527}},
}
%
%\bib{MR3121622}{article}{
%      author={De~Commer, Kenny},
%      author={Yamashita, Makoto},
%       title={Tannaka-{K}re\u\i n duality for compact quantum homogeneous
%  spaces. {I}. {G}eneral theory},
%        date={2013},
%        ISSN={1201-561X},
%     journal={Theory Appl. Categ.},
%      volume={28},
%       pages={No. 31, 1099\ndash 1138},
%      eprint={\href{http://arxiv.org/abs/1211.6552}{\texttt{arXiv:1211.6552
%  [math.OA]}}},
%      review={\MR{3121622}},
%}

\bib{DeVa10}{article}{
      author={De~Rijdt, An},
      author={Vander~Vennet, Nikolas},
       title={Actions of monoidally equivalent compact quantum groups and
  applications to probabilistic boundaries},
        date={2010},
        ISSN={0373-0956},
     journal={Ann. Inst. Fourier (Grenoble)},
      volume={60},
      number={1},
       pages={169\ndash 216},
         url={http://aif.cedram.org/item?id=AIF_2010__60_1_169_0},
      review={\MR{2664313}},
}

\bib{MR1213985}{article}{
      author={Doi, Yukio},
       title={Braided bialgebras and quadratic bialgebras},
        date={1993},
        ISSN={0092-7872},
     journal={Comm. Algebra},
      volume={21},
      number={5},
       pages={1731\ndash 1749},
         url={http://dx.doi.org/10.1080/00927879308824649},
         doi={10.1080/00927879308824649},
      review={\MR{1213985 (94a:16071)}},
}

\bib{MR4288353}{article}{
      author={Elzinga, Floris},
       title={Strong 1-boundedness of unimodular orthogonal free quantum
  groups},
        date={2021},
        ISSN={0219-0257},
     journal={Infin. Dimens. Anal. Quantum Probab. Relat. Top.},
      volume={24},
      number={2},
       pages={Paper No. 2150012, 23},
         url={http://dx.doi.org/10.1142/S0219025721500120},
         doi={10.1142/S0219025721500120},
}


\bib{MR1642584}{book}{
      author={Evans, David~E.},
      author={Kawahigashi, Yasuyuki},
       title={Quantum symmetries on operator algebras},
      series={Oxford Mathematical Monographs},
   publisher={The Clarendon Press Oxford University Press},
     address={New York},
        date={1998},
        ISBN={0-19-851175-2},
        note={Oxford Science Publications},
      review={\MR{1642584 (99m:46148)}},
}

\bib{EL}{article}{
    AUTHOR = {Exel, Ruy},
    AUTHOR = {Loring, Terry A.},
     TITLE = {Finite-dimensional representations of free product
              {$C^*$}-algebras},
   JOURNAL = {Internat. J. Math.},
  FJOURNAL = {International Journal of Mathematics},
    VOLUME = {3},
      YEAR = {1992},
    NUMBER = {4},
     PAGES = {469--476},
      ISSN = {0129-167X},
   MRCLASS = {46L05},
  MRNUMBER = {1168356},
MRREVIEWER = {Robert S. Doran},
       DOI = {10.1142/S0129167X92000217},
       URL = {https://doi.org/10.1142/S0129167X92000217},
}

\bib{MR3455859}{article}{
  author     = {Fima, Pierre},
  author     = {Vergnioux, Roland},
  journal    = {Int. Math. Res. Not. IMRN},
  title      = {A cocycle in the adjoint representation of the orthogonal free quantum groups},
  year       = {2015},
  issn       = {1073-7928},
  number     = {20},
  pages      = {10069--10094},
  doi        = {10.1093/imrn/rnu268},
  url        = {https://doi-org.ezproxy.uio.no/10.1093/imrn/rnu268},
}

\bib{MR3084500}{article}{
  author     = {Freslon, Amaury},
  journal    = {J. Funct. Anal.},
  title      = {Examples of weakly amenable discrete quantum groups},
  year       = {2013},
  issn       = {0022-1236},
  number     = {9},
  pages      = {2164--2187},
  volume     = {265},
  doi        = {10.1016/j.jfa.2013.05.037},
  url        = {https://doi-org.ezproxy.uio.no/10.1016/j.jfa.2013.05.037},
}

%\bib{MR4276316}{article}{
%      author={Gao, Li},
%      author={Harris, Samuel~J.},
%      author={Junge, Marius},
%       title={Quantum teleportation and super-dense coding in operator
%  algebras},
%        date={2021},
%        ISSN={1073-7928},
%     journal={Int. Math. Res. Not. IMRN},
%      number={12},
%       pages={9146\ndash 9179},
%         url={https://mathscinet.ams.org/mathscinet-getitem?mr=4276316},
%         doi={10.1093/imrn/rnz095},
%      review={\MR{4276316}},
%}

\bib{GWPrep}{misc}{
    author = {Gao, Li},
    author = {Weeks, John},
     title = {Furstenberg boundaries for quantum automorphism groups},
       how = {In preparation}
}

\bib{MR3801429}{article}{
      author={Hayes, Ben},
       title={1-bounded entropy and regularity problems in von neumann
  algebras},
        date={2018},
        ISSN={1073-7928},
     journal={Int. Math. Res. Not. IMRN},
      number={1},
       pages={57\ndash 137},
         url={http://dx.doi.org/10.1093/imrn/rnw237},
         doi={10.1093/imrn/rnw237},
}

\bib{arXiv:2107.03278}{misc}{
      author={Hayes, Ben},
      author={Jekel, David},
      author={Elayavalli, Srivatsav~Kunnawalkam},
       title={Property {(T)} and strong {$1$}-boundedness for von {N}eumann
  algebras},
         how={preprint},
        date={2021},
      eprint={\href{http://arxiv.org/abs/2107.03278}{\texttt{arXiv:2107.03278
  [math.OA]}}},
}

\bib{MR3391904}{article}{
  author     = {Isono, Yusuke},
  journal    = {Trans. Amer. Math. Soc.},
  title      = {Examples of factors which have no {C}artan subalgebras},
  year       = {2015},
  issn       = {0002-9947},
  number     = {11},
  pages      = {7917--7937},
  volume     = {367},
  doi        = {10.1090/tran/6321},
  url        = {https://doi-org.ezproxy.uio.no/10.1090/tran/6321},
}
%
%\bib{MR1997595}{article}{
%  author     = {Jung, Kenley},
%  journal    = {Trans. Amer. Math. Soc.},
%  title      = {The free entropy dimension of hyperfinite von {N}eumann algebras},
%  year       = {2003},
%  issn       = {0002-9947},
%  number     = {12},
%  pages      = {5053--5089},
%  volume     = {355},
%  doi        = {10.1090/S0002-9947-03-03286-0},
%  url        = {https://doi-org.ezproxy.uio.no/10.1090/S0002-9947-03-03286-0},
%}

\bib{arXiv:math/0410594}{misc}{
      author={Jung, Kenley},
       title={Some free entropy dimension inequalities for subfactors},
         how={preprint},
        date={2004},
  eprint={\href{http://arxiv.org/abs/math/0410594}{\texttt{arXiv:math/0410594
  [math.OA]}}},
}

\bib{MR2373014}{article}{
      author={Jung, Kenley},
       title={Strongly 1-bounded von Neumann algebras},
        date={2007},
        ISSN={1016-443X},
     journal={Geom. Funct. Anal.},
      volume={17},
      number={4},
       pages={1180\ndash 1200},
         url={http://dx.doi.org/10.1007/s00039-007-0624-9},
         doi={10.1007/s00039-007-0624-9},
}

\bib{arXiv:1602.04726}{misc}{
      author={Jung, Kenley},
       title={The Rank Theorem and $L^2$-invariants in Free Entropy: Global Upper Bounds},
         how={preprint},
        date={2016},
      eprint={\href{https://arxiv.org/abs/1602.04726}{\texttt{arXiv:1602.04726
  [math.OA]}}},
}

\bib{MR2464704}{article}{
      author={Kyed, David},
       title={{$L^2$}-homology for compact quantum groups},
        date={2008},
        ISSN={0025-5521},
     journal={Math. Scand.},
      volume={103},
      number={1},
       pages={111\ndash 129},
         url={http://dx.doi.org/10.7146/math.scand.a-15072},
         doi={10.7146/math.scand.a-15072},
      review={\MR{2464704}},
}

\bib{MR3683927}{article}{
      author={Kyed, David},
      author={Raum, Sven},
      author={Vaes, Stefaan},
      author={Valvekens, Matthias},
       title={{$L^2$}-{B}etti numbers of rigid {$C^*$}-tensor categories and
  discrete quantum groups},
        date={2017},
        ISSN={2157-5045},
     journal={Anal. PDE},
      volume={10},
      number={7},
       pages={1757\ndash 1791},
         url={https://doi.org/10.2140/apde.2017.10.1757},
         doi={10.2140/apde.2017.10.1757},
      review={\MR{3683927}},
}

\bib{MR3240820}{article}{
      author={Mrozinski, Colin},
       title={Quantum automorphism groups and {$SO(3)$}-deformations},
        date={2015},
        ISSN={0022-4049},
     journal={J. Pure Appl. Algebra},
      volume={219},
      number={1},
       pages={1\ndash 32},
      eprint={\href{http://arxiv.org/abs/1303.7091}{\texttt{arXiv:1303.7091
  [math.QA]}}},
         url={http://dx.doi.org/10.1016/j.jpaa.2014.04.006},
         doi={10.1016/j.jpaa.2014.04.006},
      review={\MR{3240820}},
}

%\bib{MR3426224}{article}{
%      author={Neshveyev, Sergey},
%       title={Duality theory for nonergodic actions},
%        date={2014},
%        ISSN={1867-5778},
%     journal={M{\"u}nster J. Math.},
%      volume={7},
%      number={2},
%       pages={413\ndash 437},
%      eprint={\href{http://arxiv.org/abs/1303.6207}{\texttt{arXiv:1303.6207
%  [math.OA]}}},
%      review={\MR{3426224}},
%}

\bib{MR3204665}{book}{
      author={Neshveyev, Sergey},
      author={Tuset, Lars},
       title={Compact quantum groups and their representation categories},
      series={Cours Sp{\'e}cialis{\'e}s [Specialized Courses]},
   publisher={Soci{\'e}t{\'e} Math{\'e}matique de France, Paris},
        date={2013},
      volume={20},
        ISBN={978-2-85629-777-3},
      review={\MR{3204665}},
}

\bib{NT-actions}{article}{
    author={Neshveyev, Sergey},
      author={Tuset, Lars},
     TITLE = {Deformation of {$\rm C^\ast$}-algebras by cocycles on locally
              compact quantum groups},
   JOURNAL = {Adv. Math.},
  FJOURNAL = {Advances in Mathematics},
    VOLUME = {254},
      YEAR = {2014},
     PAGES = {454--496},
      ISSN = {0001-8708},
   MRCLASS = {46L65 (20G42 46L89 81R50)},
  MRNUMBER = {3161105},
MRREVIEWER = {Kenny De Commer},
       DOI = {10.1016/j.aim.2013.12.025},
       URL = {https://doi.org/10.1016/j.aim.2013.12.025},
}


\bib{MR4198970}{article}{
  author   = {Shlyakhtenko, Dimitri},
  journal  = {J. Operator Theory},
  title    = {Von {N}eumann algebras of sofic groups with {$\beta^{(2)}_1=0$} are strongly 1-bounded},
  year     = {2021},
  issn     = {0379-4024},
  number   = {1},
  pages    = {217--228},
  volume   = {85},
  doi      = {10.7900/jot},
  url      = {https://doi-org.ezproxy.uio.no/10.7900/jot},
}

\bib{So05}{article}{
      author={So{\l}tan, Piotr~M.},
       title={Quantum {B}ohr compactification},
        date={2005},
        ISSN={0019-2082},
     journal={Illinois J. Math.},
      volume={49},
      number={4},
       pages={1245\ndash 1270},
         url={http://projecteuclid.org/euclid.ijm/1258138137},
      review={\MR{2210362}},
}

\bib{Takai}{article}{
    AUTHOR = {Takai, Hiroshi},
     TITLE = {Duality for {$C\sp{\ast} $}-crossed products and its
              applications},
 BOOKTITLE = {Operator algebras and applications, {P}art 1 ({K}ingston,
              {O}nt., 1980)},
    SERIES = {Proc. Sympos. Pure Math.},
    VOLUME = {38},
     PAGES = {369--373},
 PUBLISHER = {Amer. Math. Soc., Providence, R.I.},
      YEAR = {1982},
   MRCLASS = {46L55},
  MRNUMBER = {679719},
}


%\bib{TT20}{misc}{
%      author={Todorov, Ivan~G.},
%      author={Turowska, Lyudmila},
%       title={Quantum no-signalling correlations and non-local games},
%         how={preprint},
%        date={2020},
%      eprint={\href{http://arxiv.org/abs/2009.07016}{\texttt{arXiv:2009.07016
%  [math.OA]}}},
%}

\bib{MR2355067}{article}{
  author     = {Vaes, Stefaan},
  author     = {Vergnioux, Roland},
  journal    = {Duke Math. J.},
  title      = {The boundary of universal discrete quantum groups, exactness, and factoriality},
  year       = {2007},
  issn       = {0012-7094},
  number     = {1},
  pages      = {35--84},
  volume     = {140},
  doi        = {10.1215/S0012-7094-07-14012-2},
  url        = {https://doi-org.ezproxy.uio.no/10.1215/S0012-7094-07-14012-2},
}


\bib{MR1382726}{article}{
  author     = {Van Daele, Alfons},
  author     = {Wang, Shuzhou},
  journal    = {Internat. J. Math.},
  title      = {Universal quantum groups},
  year       = {1996},
  issn       = {0129-167X},
  number     = {2},
  pages      = {255--263},
  volume     = {7},
  doi        = {10.1142/S0129167X96000153},
  url        = {https://doi-org.ezproxy.uio.no/10.1142/S0129167X96000153},
}

\bib{MR2889142}{article}{
      author={Vergnioux, Roland},
       title={Paths in quantum {C}ayley trees and {$L^2$}-cohomology},
        date={2012},
        ISSN={0001-8708},
     journal={Adv. Math.},
      volume={229},
      number={5},
       pages={2686\ndash 2711},
         url={http://dx.doi.org/10.1016/j.aim.2012.01.011},
         doi={10.1016/j.aim.2012.01.011},
      review={\MR{2889142}},
}

\bib{MR1371236}{article}{
      author={Voiculescu, Dan},
       title={The analogues of entropy and of {F}isher's information measure in
  free probability theory. {III}. {T}he absence of {C}artan subalgebras},
        date={1996},
        ISSN={1016-443X},
     journal={Geom. Funct. Anal.},
      volume={6},
      number={1},
       pages={172\ndash 199},
         url={http://dx.doi.org/10.1007/BF02246772},
         doi={10.1007/BF02246772},
}

\bib{MR1601878}{article}{
      author={Voiculescu, Dan},
       title={A strengthened asymptotic freeness result for random matrices
  with applications to free entropy},
        date={1998},
        ISSN={1073-7928},
     journal={Internat. Math. Res. Notices},
      number={1},
       pages={41\ndash 63},
         url={http://dx.doi.org/10.1155/S107379289800004X},
         doi={10.1155/S107379289800004X},
}

\bib{MR3717094}{article}{
    AUTHOR = {Voigt, Christian},
     TITLE = {On the structure of quantum automorphism groups},
   JOURNAL = {J. Reine Angew. Math.},
    VOLUME = {732},
      YEAR = {2017},
     PAGES = {255--273},
      ISSN = {0075-4102},
       DOI = {10.1515/crelle-2014-0141},
       URL = {https://doi-org.ezproxy.uio.no/10.1515/crelle-2014-0141},
}

\bib{Wan98}{article}{
      author={Wang, Shuzhou},
       title={Quantum symmetry groups of finite spaces},
        date={1998},
        ISSN={0010-3616},
     journal={Comm. Math. Phys.},
      volume={195},
      number={1},
       pages={195\ndash 211},
         url={http://dx.doi.org/10.1007/s002200050385},
         doi={10.1007/s002200050385},
      review={\MR{1637425 (99h:58014)}},
}
%
%\bib{MR2908249}{article}{
%      author={Weaver, Nik},
%       title={Quantum relations},
%        date={2012},
%        ISSN={0065-9266},
%     journal={Mem. Amer. Math. Soc.},
%      volume={215},
%      number={1010},
%       pages={v\ndash vi, 81\ndash 140},
%         url={http://www.ams.org/books/memo/1010/},
%      review={\MR{2908249}},
%}
%
\bib{arXiv:1107.2512}{misc}{
      author={Yamashita, Makoto},
       title={Deformation of algebras associated with group cocycles},
         how={preprint},
        date={2011},
      eprint={\href{http://arxiv.org/abs/1107.2512}{\texttt{arXiv:1107.2512
  [math.OA]}}},
}


\end{biblist}
\end{bibdiv}
\end{document}